\documentclass[12pt,paper=a4, headsepline,headings=normal,abstract=true]{scrartcl}

\usepackage[english]{babel}
\usepackage[utf8]{inputenc}
\usepackage[T1]{fontenc} 
\usepackage{textcomp}
\usepackage[margin=2cm]{geometry}
\usepackage[reqno]{amsmath}  
\usepackage{amssymb}          
\usepackage{amsthm}          
\usepackage{amstext}         
\usepackage{enumerate}
\usepackage{enumitem}	  
\usepackage{bbm}
\usepackage{pifont}           
\usepackage{framed}
\usepackage{scrpage}
\usepackage{graphicx}
\usepackage{verbatim}
\usepackage{mathrsfs}
\usepackage{url}
\urldef{\urluni}{\url}{http://www.mathematik.uni-kl.de/fuana/}
\urldef{\emailgrothaus}{\url}{grothaus@mathematik.uni-kl.de}
\urldef{\emailvosshall}{\url}{vosshall@mathematik.uni-kl.de}

\setkomafont{pageheadfoot}{%
\normalfont\normalcolor\itshape\small
}
\setkomafont{pagenumber}{\normalfont\bfseries}

\makeatletter\@addtoreset{equation}{section}\makeatother

\allowdisplaybreaks

\theoremstyle{plain}      \newtheorem{theorem}{Theorem}[section]
                          \newtheorem{corollary}[theorem]{Corollary}
                          \newtheorem{proposition}[theorem]{Proposition}
													\newtheorem{condition}[theorem]{Condition}

\theoremstyle{remark}     \newtheorem{remark}[theorem]{Remark}
                          \newtheorem{lemma}[theorem]{Lemma}
													\newtheorem{example}[theorem]{Example}

\theoremstyle{definition} \newtheorem{definition}[theorem]{Definition}

\begin{document} 

\newcommand{\grad}{\nabla}
\newcommand{\D}{\partial}
\newcommand{\E}{\mathcal{E}}
\newcommand{\N}{\mathbb{N}}
\newcommand{\R}{\mathbb{R}_{\scriptscriptstyle{\ge 0}}}
\newcommand{\dom}{\mathcal{D}}
\newcommand{\ess}{\operatorname{ess~inf}}
\newcommand{\cem}{\operatorname{\text{\ding{61}}}}
\newcommand{\supp}{\operatorname{\text{supp}}}
\newcommand{\ca}{\operatorname{\text{cap}}}

\setenumerate[1]{label=(\roman*)}       
\setenumerate[2]{label=(\alph*)}

\begin{titlepage}
\title{\Large Construction and analysis of sticky reflected diffusions}
\author{\normalsize\sc Martin Grothaus \footnote{University of Kaiserslautern, P.O.Box 3049, 67653
Kaiserslautern, Germany.} \thanks{\urluni}~\thanks{\emailgrothaus}
 \and \normalsize\sc Robert Voßhall \footnotemark[2]~\thanks{\emailvosshall}}
\date{\small \today}
\end{titlepage}
\maketitle

\pagestyle{headings}

\begin{abstract}
We give a Dirichlet form approach for the construction of distorted Brownian motion in a bounded domain $\Omega$ of $\mathbb{R}^d$, $d \geq 1$, with boundary $\Gamma$, where the behavior at the boundary is sticky. The construction covers the case of a static boundary behavior as well as the case of a diffusion on the hypersurface $\Gamma$ (for $d \geq 2)$. More precisely, we consider the state space $\overline{\Omega}=\Omega \stackrel{.}{\cup} \Gamma$, the process is a diffusion process inside $\Omega$, the occupation time of the process on the boundary $\Gamma$ is positive and the process may diffuse on $\Gamma$ as long as it sticks on the boundary. The problem is formulated in an $L^2$-setting and the construction is formulated under weak assumptions on the coefficients and $\Gamma$. In order to analyze the process we assume a $C^2$-boundary and some weak differentiability conditions. In this case, we deduce that the process is also a solution to a given SDE for quasi every starting point in $\overline{\Omega}$ with respect to the underyling Dirichlet form. Under the addtional condition that $\{ \varrho =0 \}$ is of capacity zero, we prove ergodicity of the constructed process and consequently, we verify that the boundary behavior is indeed sticky. Moreover, we show ($\mathcal{L}^p$-)strong Feller properties which allow to characterize the constructed process even for every starting point in $\overline{\Omega} \backslash \{ \varrho=0\}$.
\\\\
\\
\end{abstract}

\section{Introduction}

We construct via Dirichlet form techniques diffusions on $\overline{\Omega}$ for bounded domains $\Omega$ of $\mathbb{R}^d$, $d \geq 1$, with boundary $\Gamma$ of Lebesgue measure zero, and identify them as weak solutions of
\begin{align}
d\mathbf{X}_t =& \mathbbm{1}_{\Omega} (\mathbf{X}_t) \Big( dB_t + \frac{1}{2} \nabla \ln \alpha (\mathbf{X}_t) dt \Big) - \mathbbm{1}_{\Gamma}(\mathbf{X}_t) \frac{\alpha}{\beta}(\mathbf{X}_t) ~n(\mathbf{X}_t) dt  \notag \\ 
&+ \delta ~ \mathbbm{1}_{\Gamma}(\mathbf{X}_t) \Big( dB_t^{\Gamma} + \frac{1}{2} \nabla_{\Gamma} \ln \beta (\mathbf{X}_t) dt \Big), \label{SDE} \\
dB_t^{\Gamma} =& P(\mathbf{X}_t) \circ dB_t, \notag \\
\mathbf{X}_0 =& x, \notag
\end{align}
for $x \in \overline{\Omega}$ under weak assumptions on the drifts given by $\alpha$ and $\beta$, where $n(y)$ is the outward normal at $y \in \Gamma$ and $\delta \in \{0,1\}$. In the case $\delta=1$ we additionally assume that $d \geq 2$.\\
A solution to (\ref{SDE}) can be charaterized as Brownian motion with drift inside $\Omega$ and if the process reaches $\Gamma$, Brownian motion with drift along $\Gamma$ may take place, while a further drift term is directed back into the interior of $\Omega$. In addition, the Brownian motion $B^{\Gamma}=(B_t^{\Gamma})_{t \geq 0}$ on $\Gamma$ is the projection of the $d$-dimensional Brownian motion $B=(B_t)_{t \geq 0}$ onto the Riemannian manifold $\Gamma$ (in the sense of a Stratonovich SDE). In this situation, the boundary behavior is called sticky and is connected to so-called Wentzell boundary conditions. In contrast to reflecting (Neumann) boundary conditions which provide an immediate reflection, Wentzell boundary conditions yield sojourn on $\Gamma$. The infinitesimal generator and semigroup associated to such kind of diffusions were first investigated in \cite{Fel52} and in \cite{Wen59} for more general domains. This kind of diffusion is also considered in \cite[Chap. IV, Sect. 7]{IW89} on the set ${\mathbb{R}^d_{+}:=\{ x \in \mathbb{R}^d : x_d \geq 0 \}}$, $d \geq 2$, with Lipschitz continuous drifts. However, in \cite{IW89} the diffusion on the hyperplance $\{ x \in \mathbb{R}^d: x_d=0\}$ is independent of the process inside the interior of $\mathbb{R}_{+}^d$. In \cite{Car09} the author uses a Dirichlet form approach in order to construct diffusions in a similar setting to ours with the essential difference that the boundary behavior is not sticky and also a drift does not occur. More precisely, the considered approach corresponds to ordinary reflecting boundary conditions (with the $d$-dimensional Lebesgue measure as reference measure) instead of a sticky boundary behavior. However, it is possible to switch between the setting considered in \cite{Car09} and the present setting (for $\delta=1$) using random time changes. On the other hand, in \cite{VV03} diffusion operators on $\overline{\Omega}$ with sticky boundary behavior are considered, but without introducing a boundary operator on $\Gamma$. This is in accordance with our setting for $\delta=0$. Nevertheless, we also construct and analyze the underlying dynamics. Moreover, neither in \cite{Car09} nor in \cite{VV03} Feller properties of the associated semigroup are investigated. \\
In \cite{Gra88} a sticky diffusion is constructed by probabilistic methods. The constructed process coincides with our setting in special cases, but the considered domain is determined by the zero set of a $C^2(\mathbb{R}^d)$-function and the drift is assumed to be Lipschitz continuous and bounded which is quite restrictive. The author uses the constructed diffusion to model a system of particles interacting at the boundary. This interacting particle system in turn is used to study the behavior of molecules in a chromatography tube. This application as well as the application given in \cite{FGV14} and \cite{GV14} motivates our present considerations.\\
Relating to the construction of sticky reflection the authors in \cite{VV03} remark the following: "If one wants to describe particles which may
temporarily concentrate on the boundary, the reference measure should offer this
possibility - meaning that the boundary should not be a null set". Accordingly, we consider the reference measure $\varrho~(\lambda + \sigma)$ in order to assign mass to the boundary $\Gamma$, where $\lambda$ denotes the Lebesgue measure on $\overline{\Omega}$ and $\sigma$ the surface measure on $\Gamma$. Additionally, we point out the connection of this approach to random time changes.\\
In \cite{EP12} the authors analyze Brownian motion on $[0, \infty)$ which is sticky in $0$. They show that strong solutions do not exist and that the sticky Brownian motion is the limit of time scaled reflected Brownian motions. This suggests that a strong solution in our framework also does not exist and hence, the solutions we construct in this paper are optimal in this sense. \\

In Section 2, notations are explained and some basic facts about manifolds and Brownian motion on manifolds are stated. In Section 3, we construct a diffusion process via a Dirichlet form approach on sets $\overline{\Omega}$ with Lipschitz boundary and with very singular drift terms. For a $C^2$-boundary and under additional assumptions on the density $\varrho$, this process is identified as solution of an associated martingale problem und finally, as solution of the SDE (\ref{SDE}) for all starting points (except a set of capacity zero with respect to the underlying Dirichlet form) in Section 4. Moreover, we prove an ergodic theorem in order to verify that the boundary is indeed sticky. In Section 5, strong Feller properties of the underlying resolvent are established and used in order to strenghten the preceding results such that the constructed process solves (\ref{SDE}) for every starting point.\\

The main results are formulated in Theorem \ref{thmsolSDE}, Theorem \ref{thmergo} and Theorem \ref{main}.

\section{Preliminaries} \label{sectprel}

\subsection{General notation}

Throughout this paper, $\Omega \subset \mathbb{R}^d$, $d \geq 1$, denotes a nonempty bounded domain such that its boundary $\Gamma:=\partial \Omega$ is of Lebesgue measure zero. In the case $\delta=1$ we assume that $d \geq 2$. The standard scalar product in $\mathbb{R}^d$ is given by $(\cdot,\cdot)$ and norms in $\mathbb{R}^d$ by $|\cdot|$ (in particular, for the modulus in $\mathbb{R}$; eventually labeled by a lower index in order to distinguish norms). Similarly, $\Vert \cdot \Vert$ denotes norms in function spaces. The metric on $\mathbb{R}^d$ induced by the euclidean metric is denoted by $d_{\text{euc}}$. \\

For smooth functions, we denote by $\nabla$ the gradient as well as the Jacobian in the case of vector valued functions. Let $\nabla_i=\partial_i$, $i=1,\dots,d$, be the partial derivatives with respect to cartesian coordinates. If we take partial derivatives and want to point out the underlying coordinates, we write for example $\frac{\partial}{\partial x_i}$. Moreover, $\nabla^2$ denotes the Hessian for functions mapping from subsets of $\mathbb{R}^d$ to $\mathbb{R}$ and $\Delta= \text{Tr}( \nabla^2 )$ the Laplacian. In the case of Sobolev functions we use the same notations in the weak sense.

\subsection{Submanifolds in the euclidean space}

We recall some facts about hypersurfaces of $\mathbb{R}^d$ and Riemannian geometry:

\begin{definition} \label{defbdry}
Let $\Omega \subset \mathbb{R}^d$ be a bounded domain. The boundary $\Gamma$ of $\Omega$ is said to be \mbox{\textbf{Lipschitz continuous}} (respectively \textbf{$C^k$-smooth}) if it is locally the graph of a Lipschitz continuous (respectively $C^k$-) function, i.e., for every $x \in \Gamma$ exists a neighborhood $V$ of $x$ in $\mathbb{R}^d$ and new orthogonal coordinates $(y_1,\dots,y_d)$ (given by an orthogonal map $T$) such that
\begin{enumerate}
\item $V$ is a hypercube in the new coordinates:
\[ V=\{(y_1,\dots,y_d) : -a_i < y_i < a_i,~ 1 \leq i \leq d \}; \]
for some $a_i >0$, $1 \leq i \leq d$.
\item there exists a Lipschitz continuous (respectively $C^k$-) function $\varphi$, defined on 
\[ V^{\prime}=\{y^{\prime}=(y_1,\dots, y_{d-1}) : -a_i < y_i < a_i , ~ 1 \leq i \leq d-1 \} \]
and such that 
\begin{align*}
|\varphi(y^{\prime})| \leq \frac{a_d}{2} \text{ for every } y^{\prime} \in V^{\prime},\\
\Omega \cap V = \{ y=(y^{\prime},y_d) \in V : y_d < \phi(y^{\prime}) \},\\
\Gamma \cap V = \{ y=(y^{\prime},y_d) \in V : y_d = \phi(y^{\prime}) \}.
\end{align*}
\end{enumerate}
\end{definition}

So $\Gamma$ is Lipschitz continuous (respectively $C^k$-smooth) if $\Omega$ is locally below the graph of a Lipschitz continuous (respectively $C^k$-) function and the graph coincides with $\Gamma$. In this case, we also simply say that $\Gamma$ is Lipschitz (respectively $C^k$) or that $\Omega$ has Lipschitz boundary (respectively $C^k$-boundary).

\begin{remark} \label{remimpl}
Definition \ref{defbdry} makes $\Gamma$ a hypersurface of $\mathbb{R}^d$ and $\Gamma$ is in the coordinates $(y_1,\dots,y_d)$ locally  implicitly given by $F(y_1,\dots,y_d)=0$, where $F(y_1,\dots,y_d):= y_d - \varphi(y_1,\dots,y_{d-1})$ defined on $V$. Occasionally, we may refer to $y^{\prime}=(y_1,\dots,y_{d-1})$ as local coordinates of $\Gamma$. Note that in the case of a Lipschitz continuous boundary $\Gamma$, $\varphi$ is almost everywhere differentiable by Rademacher's theorem. Hence, $F$ is almost everywhere differentiable.
One can also consider $\overline{\Omega}$ as $d$-dimensional manifold with boundary.
\end{remark}


\begin{remark} \label{remriemann} 
The above definition makes $\Gamma$ a Riemannian manifold with induced Riemannian metric locally given by $G=(g_{ij})_{ij}$ where
\[ g_{ij}(y^{\prime})=\left\{\begin{array}{cl} \frac{\partial \varphi}{\partial y_i}(y^{\prime})~ \frac{\partial \varphi}{\partial y_j}(y^{\prime}) & \mbox{if } i \neq j\\
1+ \big(\frac{\partial \varphi}{\partial y_i}(y^{\prime})\big)^2 & \mbox{if } i=j \end{array}\right. \]
for $1 \leq i,j \leq d-1$. Denote by $g$ the determinant and by $(g^{ij})_{ij}$ the inverse $G^{-1}$ of $G$. Note that the surface measure $\sigma$ on $\Gamma$ is given by the $(d-1)$-dimensional Hausdorff measure which can locally be written as
\[ d\sigma= \sqrt{g} ~dy_1 \cdots dy_{d-1}= \sqrt{g} ~dy^{\prime}. \]
\end{remark}

\begin{definition} \label{defnormal}
Let $\Omega$ be open and bounded with Lipschitz continuous boundary $\Gamma$ and let $F$ be as in Remark \ref{remimpl}. Then we define for $y=(y_1,\dots,y_d)$ in $ V$
\[ \tilde{n}(y):= \frac{\nabla F(y) }{|\nabla F(y) |} = \frac{(- \nabla \varphi(y^{'}), 1)}{\sqrt{|\nabla \varphi(y^{'})|^2}+1} \]
supposed that $\varphi$ is differentiable at $(y_1,\dots,y_{d-1})$. Let $x \in \Gamma$ and $T \in \mathbb{R}^{d \times d}$ be the orthogonal coordinate transformation from Definition \ref{defbdry}. Then define the (outward) \textbf{normal vector} at $x$ by
\[ n(x):=  T^{-1}~ \tilde{n}(Tx). \]
\end{definition}

\begin{remark}
Note that the definition of $n$ also makes sense in a neighborhood of $x$ and $n$ is differentiable near $x$ if $\Gamma$ is $C^2$. 
\end{remark}

\begin{definition} \label{defproj}
Let $x \in \Gamma$ be such that $n(x)$ exists in the sense of Definition \ref{defnormal}. Define
\[ P(x):= E- n(x)n(x)^t \in \mathbb{R}^{d \times d}, \]
where $E$ is the $d \times d$ identity matrix. We call $P(x)$ the \textbf{orthogonal projection on the tangent space} at $x$. Note that $P(x)z=z- (n(x),z)~ n(x)$ for $z \in \mathbb{R}^d$.
\end{definition}

\begin{definition}
Let $f \in C^1(\overline{\Omega})$ and $x \in \Gamma$. Then we define (whenever $\Gamma$ is sufficiently smooth at $x$) the \textbf{gradient} of $f$ at $x$ along $\Gamma$ by 
\[ \nabla_{\Gamma} f (x):= P(x) \nabla f(x) \]
and if $f \in C^2(\overline{\Omega})$ the \textbf{Laplace-Beltrami} of $f$ at $x$ by
\[ \Delta_{\Gamma} f(x) := \text{Tr} ( \nabla_{\Gamma}^2 f(x))= \text{div}_{\Gamma} \nabla_{\Gamma} f(x)= \text{Tr}(P(x) \nabla (P(x) \nabla f(x))), \]
where $\text{div}_{\Gamma} \Phi := \text{Tr}(P \nabla \Phi)$ for $\Phi=(\Phi_1,\dots,\Phi_d) \in C^1(\overline{\Omega};\mathbb{R}^d)$ with $\nabla \Phi=J \Phi=(\nabla \Phi_1|\dots| \nabla \Phi_d)$.
Analogously, we define higher derivatives of order $k \in \mathbb{N}$. In this way, let $C^k(\Gamma_0)$ be the space of continuously differentiable functions on $\Gamma_0$ obtained by restriction of $C^k(\overline{\Omega})$-functions, where $\Gamma_0$ is an open subset of $\Gamma$ in the subspace topology. As usual, set $C^{\infty}(\Gamma_0):=\cap_{k \in \mathbb{N}}~ C^k(\Gamma_0)$. Moreover, in the case that $n$ is differentiable at $x$ we define the \textbf{mean curvature} of $\Gamma$ at $x$ by 
\[ \kappa(x):= \text{div}_{\Gamma}~ n (x). \]
\end{definition}

\begin{remark}
\begin{enumerate}
\item For $f \in C^k(\overline{\Omega})$, $k \in \mathbb{N}$, the above definitions are in accordance with the ordinary definition on Riemannian manifolds for $f|_{\Gamma}$ using the inclusion $T_x \Gamma \hookrightarrow \mathbb{R}^d$, where $T_x \Gamma$ denotes the tangent space at $x \in \Gamma$. Conversely, if a function $f$, defined only on the Riemannian manifold $\Gamma$, is in $C^k(\Gamma)$, $k \in \mathbb{N}$, in the sense of manifolds, it is possible to extend the definition of $f$ to a $C^k$-function on an open set in $\mathbb{R}^d$ which contains $\Gamma$ and then it is feasible to use the definitions given above. Thus, $C^k(\Gamma)$, $k \in \mathbb{N}$, contains exactly the functions on $\Gamma$ obtained by restricting functions from $C^k(\overline{\Omega})$ to $\Gamma$.
\item By the Riemannian metric $G$ mentioned in Remark \ref{remriemann}, we have in local coordinates (using the Einstein summation convention) for the Riemannian gradient the representation
$ g^{ij} \frac{\partial f}{\partial y_i} \frac{\partial}{\partial y_j} $
and for the Laplace-Beltrami operator
$ \frac{1}{\sqrt{g}} \frac{\partial}{\partial y_i} \big( \sqrt{g} g^{ij} \frac{\partial f}{\partial y_j} \big). $
\item For smooth functions, we have the divergence theorem
\begin{align} \label{divergence} \int_{\Gamma} (\Phi , \nabla_{\Gamma} g)~ d\sigma = - \int_{\Gamma} \text{div}_{\Gamma} \Phi~ g~d\sigma,
\end{align}
where $\Phi$ is $\mathbb{R}^d$-valued (see e.g. \cite[Chap. 2, Proposition 2.2]{Tay11}).
\end{enumerate}
\end{remark}

\begin{lemma} \label{lemcurv}
Assume that $\Gamma$ is $C^2$-smooth. Then
\[ ( P \nabla )^t P = - \kappa n. \] 
\end{lemma}

\begin{proof}
Fix $i \in \{1, \dots d \}$. It holds
\begin{align*}
 \big( (P \nabla)^t P \big)_i &=  \sum_{k,j} P_{jk} \nabla_j P_{ik} \\
&= - \sum_{k,j} P_{jk} \nabla_j (n_i n_k) \\
&= - \sum_{k,j} (1-n_j n_k) (\nabla_j n_i n_k + n_i \nabla_j n_k)\\
&= - \big( \sum_{k,j} (1-n_j n_k) \nabla_j n_k \big) n_i -  \sum_{k,j} (1-n_j n_k) \nabla_j n_i n_k \\
&= - \text{Tr} ( P \nabla n ) n_i - (n, P \nabla n_i) = - \kappa n_i - (n, P\nabla n_i).
\end{align*}
Using that $P$ is the orthogonal projection on $\big(\text{span}(n) \big)^{\perp}$, we get that $(n, P\nabla n_i)=0$ and therfore, the assertion holds true.
\end{proof}

\begin{definition}
Let $\Gamma_0$ be an open subset of $\Gamma$ in the subspace topology. The \textbf{Sobolev space} $H^{1,k}(\Gamma_0)$, $k \geq 1$, is defined by $\overline{C^1(\Gamma_0)}^{\Vert \cdot \Vert_{H^{1,k}(\Gamma_0)}} \subset L^k(\Gamma_0;\sigma)$, i.e., the closure $C^1(\Gamma_0)$ with respect to the norm
\[ \Vert \cdot \Vert_{H^{1,k}(\Gamma_0)} := \big( \Vert \cdot \Vert_{L^k(\Gamma_0;\sigma)}^k + \Vert \nabla_{\Gamma} \cdot \Vert_{L^k(\Gamma_0;\sigma)}^k \big)^{\frac{1}{k}}. \]
\end{definition}

\begin{remark}
$H^{1,k}(\Gamma_0)$ can also be charaterized as the space of functions which are in local coordinates in the corresponding Sobolev space.\\
 If $f \in H^{1,k}(\Gamma_0)$ and $(f_n)_{n \in \mathbb{N}}$ is an approximating sequence of smooth functions, Cauchy in $H^{1,k}(\Gamma_0)$, we call the $L^k(\Gamma_0;\sigma)$-limit of $(\nabla_{\Gamma} f_n)_{n \in \mathbb{N}}$ the weak gradient of $f$ and denote it by $\nabla_{\Gamma} f$. In the case $\Gamma_0=\Gamma$, (\ref{divergence}) transfers from $f_n$ to $f$ using a continuity argument provided that $\Phi \in L^{k^{\prime}}(\Gamma;\sigma)$ for $\frac{1}{k}+\frac{1}{k^{\prime}}=1$.
\end{remark}

\subsection{Brownian motion on manifolds}

We shortly recall some facts about Brownian motion on $\Gamma$. For details about stochastic analysis on manifolds, we refer to \cite{HT94}, \cite{Hsu02} and \cite{IW89}.\\

By definition, Brownian motion $(B_t^{\Gamma})_{t \geq 0}$ on $\Gamma$ is a $\Gamma$-valued stochastic process that is generated by $\frac{1}{2} \Delta_{\Gamma}$, in analogy to Brownian motion on $\mathbb{R}^d$, in the sense that $(B_t^{\Gamma})_{t \geq 0}$ solves the martingale problem for $(\frac{1}{2} \Delta_{\Gamma},C^{\infty}(\Gamma))$. We recall the following:

\begin{lemma} 
Let $\Gamma$ be a submanifold of $\mathbb{R}^d$ as in Definition \ref{defbdry}. Then a solution of the Stratonovich SDE
\[ d\mathbf{X}_t= P(\mathbf{X}_t) \circ dB_t, \ \ \mathbf{X}_0 \in \Gamma, \]
is a Brownian motion on $\Gamma$, where $(B_t)_{t \geq 0}$ is a Brownian motion in $\mathbb{R}^d$.
\end{lemma}

\begin{proof}
See \cite[Chap. 3, Sect. 2]{Hsu02}.
\end{proof}

\begin{remark}
Note that the dimension of the driving Brownian motion $(B_t)_{t \geq 0}$ is strictly larger than the dimension of the submanifold $\Gamma$ and hence, according to \cite{Hsu02} the driving Brownian motion contains some extra information beyond what is usually provided by a Brownian motion on $\Gamma$. Furthermore, a solution of the above SDE is naturally $\Gamma$-valued, since $P(x)z$ is tangential to $\Gamma$ at $x$ for every $x \in \Gamma$ and $z \in \mathbb{R}^d$. In our application, it is natural to construct a Brownian motion on $\Gamma$ by means of a $d$-dimensional Brownian motion, since a Brownian motion on $\mathbb{R}^d$ is involved anyway.
\end{remark}

We also need the following result:

\begin{theorem}[It{\^o}-Stratonovich transformation rule]
Consider a diffusion in $\mathbb{R}^d$ driven by a $d$-dimensional Brownian motion $B=(B_t)_{t \geq 0}$ via the Stratonovich SDE
\[ d\mathbf{X}_t = S(\mathbf{X}_t) \circ dB_t, \]
where $S:\mathbb{R}^d \mapsto \mathbb{R}^{d \times d}$ is $C^1$-smooth and symmetric. Then the It{\^o} form reads
\[ d\mathbf{X}_t = S(\mathbf{X}_t) dB_t + \frac{1}{2} \big((S \nabla)^t S \big)(\mathbf{X}_t) dt. \]
\end{theorem}


\section{The Dirichlet form and the associated Markov process}\label{sectprocess}

\begin{condition} \label{conddensity}
$\Gamma$ is Lipschitz continuous. Moreover, $\alpha \in L^1(\Omega; \lambda)$, $\alpha >0$ $\lambda$-a.e., and $\beta \in L^1(\Gamma; \sigma)$, $\beta >0$ $\sigma$-a.e..
\end{condition}

Define
\[ \varrho:= \mathbbm{1}_{\Omega} ~ \alpha + \mathbbm{1}_{\Gamma} ~ \beta \]
as well as
\[ \mu:= \varrho ~ (\lambda + \sigma) = \alpha \lambda + \beta  \sigma. \]
Note that the condition $\alpha \in L^1(\Omega; \lambda)$, $\alpha >0$ $\lambda$-a.e., and $\beta \in L^1(\Gamma; \sigma)$, $\beta >0$ $\sigma$-a.e. is equivalent to $\varrho \in L^1(\overline{\Omega};\lambda + \sigma)$, $\varrho >0$ $(\lambda + \sigma)$-a.e..

\begin{proposition}
Under Condition \ref{conddensity} we have that $C^{\infty}(\overline{\Omega})$ is dense in $L^2(\overline{\Omega};\mu)$.
\end{proposition}

\begin{proof}
Let $f \in L^2(\overline{\Omega};\mu)$. By \cite[Corollary 7.5.5]{Bau81} we get that $C(\overline{\Omega})$ is dense in $L^2(\overline{\Omega};\mu)$, since $\mu$ is a Baire measure and $\overline{\Omega}$ is compact. Hence there exists a sequence $(g_i)_{i \in \mathbb{N}}$ in $C(\overline{\Omega})$ converging to $f$ with respect to $\Vert \cdot \Vert_{L^2(\overline{\Omega};\mu)}$. Due to the Stone-Weierstra{\ss} theorem  $C^{\infty}(\overline{\Omega})$ is dense in $C(\overline{\Omega})$ with respect to $\Vert \cdot \Vert_{\sup}$, where $\Vert h \Vert_{\sup} := \sup_{x \in \overline{\Omega}} |h(x)|$ for $h \in C(\overline{\Omega})$. Thus, for each $i \in \mathbb{N}$ exists a sequence $(f_j^i)_{j \in \mathbb{N}}$ in $C^{\infty}(\overline{\Omega})$ converging to $g_i$ with respect to $\Vert \cdot \Vert_{\sup}$. \\
Let $\varepsilon >0$. Then, by the previous considerations, there exists some $k \in \mathbb{N}$ such that \linebreak $\Vert f - g_k \Vert_{L^2(\overline{\Omega};\mu)} < \frac{\varepsilon}{2}$. Accordingly, there exists some $l \in \mathbb{N}$ such that $\Vert g_k - f_l^k \Vert_{\sup} < \frac{\varepsilon}{2 \sqrt{\mu (\overline{\Omega})}}$. Hence,
\begin{align*}
\Vert f - f_l^k \Vert^2_{L^2(\overline{\Omega};\mu)} &\leq 2 \Vert f-g_k \Vert^2_{L^2(\overline{\Omega};\mu)} + 2 \Vert g_k - f_l^k \Vert^2_{L^2(\overline{\Omega};\mu)} \\
&\leq \frac{\varepsilon^2}{2} + 2 \Vert g_k - f_l^k \Vert_{\sup}^2~ \mu(\overline{\Omega}) \\
&< \frac{\varepsilon^2}{2} + \frac{\varepsilon^2}{2 \mu(\overline{\Omega})}  \mu(\overline{\Omega}) < \varepsilon^2.
\end{align*}
Therefore, $C^{\infty}(\overline{\Omega})$ is dense in $L^2(\overline{\Omega};\mu)$.
\end{proof}


Let the symmetric and positive definite bilinear form $(\mathcal{E},\mathcal{D})$ be given by
\begin{align} \label{defform} \mathcal{E}(f,g):= \frac{1}{2} \int_{\Omega} (\nabla f, \nabla g)~ \alpha d\lambda + \frac{\delta}{2} \int_{\Gamma} (\nabla_{\Gamma} f,\nabla_{\Gamma} g) ~\beta d\sigma \ \text{ for } f,g \in \mathcal{D}:=C^1(\overline{\Omega}), \end{align}
where $(\cdot,\cdot)$ denotes the euclidean scalar product in $\mathbb{R}^d$ and $\delta \in \{0,1\}$. In addition, let 
\[ \mathcal{E}_{\Omega}(f,g):=  \frac{1}{2} \int_{\Omega} (\nabla f, \nabla g)~ \alpha d\lambda \ \text{ for } f,g \in \mathcal{D}_{\Omega}:=C^1(\overline{\Omega}) \]
as well as
\[ \mathcal{E}_{\Gamma}(f,g):= \frac{1}{2} \int_{\Gamma} (\nabla_{\Gamma} f,\nabla_{\Gamma} g) ~\beta d\sigma \ \text{ for } f,g \in \mathcal{D}_{\Gamma}:=C^1(\Gamma). \]
Note that $e(\mathcal{D})=e(\mathcal{D}_{\Omega})= \mathcal{D}_{\Gamma}$, where $e: C^1(\overline{\Omega}) \rightarrow C^1(\Gamma)$ is defined by the restriction of functions to $\Gamma$. In this terms, for $f,g \in \mathcal{D}$ we get 
\[ \mathcal{E}(f,g)= \mathcal{E}_{\Omega}(f,g) + \delta ~\mathcal{E}_{\Gamma}(f,g).\]

In order to prove closability of $(\mathcal{E},\mathcal{D})$, we need an additional assumption on the density $\varrho$. Define
\[ R_{\alpha}(\Omega) :=\{ x \in \Omega : \int_{\{ y \in \Omega : |x-y| < \epsilon \}} \alpha^{-1} d\lambda < \infty \ \text{ for some } \epsilon >0 \} \]
and analogously $R_{\beta}(\Gamma)$ with $\Omega$ replaced by $\Gamma$ and $\lambda$ replaced by $\sigma$. 

\begin{condition}[Hamza condition] \label{condhamza}
$\alpha =0$ $\lambda$-a.e. on $\Omega \backslash R_{\alpha}(\Omega)$ and additionally $\beta=0$ $\sigma$-a.e. on $\Gamma \backslash R_{\beta}(\Gamma)$ if $\delta=1$.
\end{condition}

\begin{lemma} \label{lemclosable}
Assume that Condition \ref{conddensity} and Condition \ref{condhamza} are fulfilled. Then the densely defined, symmetric bilinear forms $(\mathcal{E}_{\Omega},\mathcal{D}_{\Omega})$ and $(\mathcal{E}_{\Gamma},\mathcal{D}_{\Gamma})$ (if $\delta=1$) are closable on $L^2(\overline{\Omega};\alpha \lambda)$ and on $L^2(\Gamma; \beta \sigma)$ respectively. Moreover, the closures $(\mathcal{E}_{\Omega},D(\mathcal{E}_{\Omega}))$ and $(\mathcal{E}_{\Gamma},D(\mathcal{E}_{\Gamma}))$ are conservative, strongly local, regular, symmetric Dirichlet forms. 
\end{lemma}

\begin{proof}
The symmetric densely defined bilinear forms are closable and its closures are symmetric Dirichlet forms by \cite[Chap. 2, Sect. 2, Example a)]{MR92} (see in particular Remark 2.3 of the reference). The remaining properties follow exactly like in the following proofs for the closure of $(\mathcal{E},\mathcal{D})$.
\end{proof}

\begin{proposition}
Suppose that Condition \ref{conddensity} and Condition \ref{condhamza} are satisfied. Then $(\mathcal{E},\mathcal{D})$ is closable on $L^2(\overline{\Omega};\mu)$. We denote its closure by $(\mathcal{E},D(\mathcal{E}))$.
\end{proposition}

\begin{proof}
Let $(f_k)_{k \in \mathbb{N}}$ be a Cauchy sequence in $\mathcal{D}=C^1(\overline{\Omega})$ with respect to $\mathcal{E}$, i.e., 
\[ \mathcal{E}(f_k-f_l,f_k-f_l) \rightarrow 0 \ \text{ as } k,l \rightarrow \infty. \]
Moreover, assume that $(f_k)_{k \in \mathbb{N}}$ converges to $0$ in $L^2(\overline{\Omega};\mu)$. We have to show that $\mathcal{E}(f_k,f_k) \rightarrow 0$ as $k \rightarrow \infty$.\\
Since $(f_k)_{k \in \mathbb{N}}$ is a Cauchy sequence with respect to $\mathcal{E}$, it is also a Cauchy sequence with respect to $\mathcal{E}_{\Omega}$ (and $\mathcal{E}_{\Gamma}$ if $\delta=1$). Moreover, the convergence of  $(f_k)_{k \in \mathbb{N}}$ to $0$ in $L^2(\overline{\Omega};\mu)$ implies by definition the convergence to $0$ in $L^2(\overline{\Omega};\alpha \lambda)$ and $L^2(\Gamma;\beta \sigma)$. Therefore, we get $\mathcal{E}_{\Omega}(f_k,f_k) \rightarrow 0$ (and ${\mathcal{E}_{\Gamma}(f_k,f_k) \rightarrow 0}$ if $\delta=1$) as $k \rightarrow \infty$ by Lemma \ref{lemclosable}. Hence,
\[ \mathcal{E}(f_k,f_k) = \mathcal{E}_{\Omega}(f_k,f_k) + \delta ~\mathcal{E}_{\Gamma}(f_k,f_k) \rightarrow 0 \ \text{ as } k \rightarrow \infty. \]
\end{proof}

\begin{proposition}
Suppose that Condition \ref{conddensity} and Condition \ref{condhamza} are satisfied. Then $(\mathcal{E},D(\mathcal{E}))$ is a symmetric, regular Dirichlet form.
\end{proposition}

\begin{proof}
The Markov property follows as in \cite[Chap.2, Sect. 2, Example c)]{MR92} by \cite[Chap. 1, Proposition 4.10]{MR92} and the chain rule. Hence, $(\mathcal{E},D(\mathcal{E}))$ is a symmetric Dirichlet Form. By the Stone-Weierstra{\ss} theorem, it holds that $C^{\infty}(\overline{\Omega})$ is dense in $C(\overline{\Omega})$ with respect to $\Vert \cdot \Vert_{\sup}$. Furthermore, $\mathcal{D}$ is dense in $D(\mathcal{E})$ with respect to the $\mathcal{E}_1$-norm. Since $C^{\infty}(\overline{\Omega}) \subset \mathcal{D} \subset D(\mathcal{E}) \cap C(\overline{\Omega})$, we obtain that $(\mathcal{E},D(\mathcal{E}))$ is also regular.
\end{proof}


\begin{proposition}
Suppose that Condition \ref{conddensity} and Condition \ref{condhamza} are satisfied. Then the symmetric, regular Dirichlet form $(\mathcal{E},D(\mathcal{E}))$ is strongly local and recurrent.
\end{proposition}

\begin{proof}
Using \cite[Theo. 3.1.1]{FOT94} and \cite[Exercise 3.1.1]{FOT94} it is sufficient to show the strong local property for elements in $\mathcal{D}$. Therefore, let $f,g \in \mathcal{D}$ such that $g$ is constant on some open neighborhood $U$ of $\text{supp}(f)$ (in the trace topology of $\overline{\Omega}$). Then
\begin{align*}
\mathcal{E}(f,g)=& \frac{1}{2} \int_{\Omega} (\nabla f, \nabla g) ~ \alpha d\lambda + \frac{\delta}{2} \int_{\Gamma} (\nabla_{\Gamma} f, \nabla_{\Gamma} g) ~ \beta d\sigma \\
=& \frac{1}{2} \int_{\Omega \cap \text{supp}(f)} (\nabla f, \nabla g) ~ \alpha d\lambda + \frac{1}{2} \int_{\Omega \backslash \text{supp}(f)} (\nabla f, \nabla g) ~ \alpha d\lambda\\
 &+ \frac{\delta}{2} \int_{\Gamma \cap \text{supp}(f)} (\nabla_{\Gamma} f, \nabla_{\Gamma} g) ~ \beta  d\sigma + \frac{\delta}{2} \int_{\Gamma \backslash \text{supp}(f)} (\nabla_{\Gamma} f, \nabla_{\Gamma} g) ~ \beta d\sigma \\
=& 0,
\end{align*}
because each summand is zero, since the integrals are defined over sets where either $f$ or $g$ is constant. Hence, $(\mathcal{E},D(\mathcal{E}))$ is strongly local.
Clearly, $\mathbbm{1}_{\overline{\Omega}} \in \mathcal{D} \subset D(\mathcal{E})$ and $\mathcal{E}(\mathbbm{1}_{\overline{\Omega}},\mathbbm{1}_{\overline{\Omega}})=0$. Therefore, $(\mathcal{E},D(\mathcal{E}))$ is also recurrent.
\end{proof}

By summarizing the preceding results, we get the following theorem:

\begin{theorem}
Assume Condition \ref{conddensity} and Condition \ref{condhamza}. Then the symmetric and positive definite bilinear form $(\mathcal{E},D)$ is denesly defined and closable on $L^2(\overline{\Omega};\mu)$. Its closure $(\mathcal{E},D(\mathcal{E}))$ is a recurrent, strongly local, regular, symmetric Dirichlet form on $L^2(\overline{\Omega};\mu)$.
\end{theorem}

By the theory of Dirichlet forms, we obtain immediately the following theorem.  For details see e.g. \cite[Chap. V, Theorem 1.11]{MR92} or \cite[Theorem 7.2.2 and Exercise 4.5.1]{FOT94}. We remark that the definitions of capacities (and hence, of exceptional sets) used in the textbooks \cite{FOT94} and \cite{MR92} are introduced in different ways, but that the defintions coincide in our setting (see \cite[Chap. III, Remark 2.9 and Exercise 2.10]{MR92}). $(T_t)_{t >0}$ denotes the sub-Markovian strongly continuous contraction semigroup on $L^2(\overline{\Omega};\mu)$ corresponding to $(\mathcal{E},D(\mathcal{E}))$.

\begin{theorem} \label{thmdiff}
Suppose that Condition \ref{conddensity} and Condition \ref{condhamza} are satisfied. Then there exists a conservative diffusion process (i.e. a strong Markov process with continuous sample paths and infinite life time)
\[ \mathbf{M}:=\big( \mathbf{\Omega}, \mathcal{F}, (\mathcal{F}_t)_{t \geq 0}, (\mathbf{X}_t)_{t \geq 0}, (\Theta_t)_{t \geq 0}, (\mathbf{P}_x)_{x \in \overline{\Omega}} \big) \]
with state space $\overline{\Omega}$ which is properly associated with $(\mathcal{E},D(\mathcal{E}))$, i.e., for all ($\mu$-versions of) $f \in \mathcal{B}_b(\overline{\Omega}) \cap L^2(\overline{\Omega};\mu)$ and all $t >0$ the function
\[ \overline{\Omega} \ni x \mapsto p_t f(x):= \mathbb{E}_x \big(f(\mathbf{X}_t) \big) := \int_{\Omega} f(\mathbf{X}_t) d\mathbf{P}_x \in \mathbb{R} \]
is a quasi continuous version of $T_t f$. $\mathbf{M}$ is up to $\mu$-equivalence unique. In particular, $\mathbf{M}$ is $\mu$-symmetric, i.e.,
\[ \int_{\overline{\Omega}} p_t f ~ g ~d\mu = \int_{\overline{\Omega}} f ~ p_t g ~d\mu \ \text{ for all } f,g \in \mathcal{B}_b(\overline{\Omega}) \ \text{ and all } t >0, \]
and has $\mu$ as invariant measure, i.e.,
\[ \int_{\overline{\Omega}} p_t f~ d\mu = \int_{\overline{\Omega}} f~ d\mu \ \text{ for all } f \in \mathcal{B}_b(\overline{\Omega}) \ \text{ and all } t >0. \]
\end{theorem}

\begin{remark} \label{remcanonical}
Note that $\mathbf{M}$ is canonical, i.e., $\mathbf{\Omega}=C(\mathbb{R}_{+},\overline{\Omega})$ and $\mathbf{X}_t(\omega)=\omega(t)$, $\omega \in \mathbf{\Omega}$. 
For each $t \geq 0$ we denote by $\Theta_t:\mathbf{\Omega} \rightarrow \mathbf{\Omega}$ the shift operator defined by $\Theta_t(\omega)=\omega(\cdot +t)$ for $\omega \in \mathbf{\Omega}$ such that $\mathbf{X}_s \circ \Theta_t = \mathbf{X}_{s+t}$ for all $s \geq 0$.
We take into account to extend the setting to $C(\mathbb{R}_{+},\mathbb{R}^d)$ by neglecting paths leaving $\overline{\Omega}$.
\end{remark}

\section{Analysis of the Markov process} \label{sectanapro} 

\subsection{Generators and boundary conditions}

By Friedrichs representation theorem we have the existence of a unique self-adjoint generator $(L,D(L))$ corresponding to $(\mathcal{E},D(\mathcal{E}))$.

\begin{proposition}
Suppose that Condition \ref{conddensity} and Condition \ref{condhamza} are satisfied. Then there exists a unique, positive, self-adjoint, linear operator $(L,D(L))$ on $L^2(\overline{\Omega}; \mu)$ such that
\[ D(L) \subset D(\mathcal{E}) \ \text{ and } \ \mathcal{E}(f,g)=(-Lf,g)_{L^2(\overline{\Omega};\mu)} \ \text{ for all } f \in D(L), \ g \in D(\mathcal{E}).\]
\end{proposition}

\begin{condition} \label{condcont}
$\Gamma$ is $C^2$-smooth. Moreover, $\alpha, \beta \in C(\overline{\Omega})$, $\alpha >0$ $\lambda$-a.e. on $\Omega$, $\beta >0$ $\sigma$-a.e. on $\Gamma$ such that $\sqrt{\alpha} \in H^{1,2}(\Omega)$ and additionally $\sqrt{\beta} \in H^{1,2}(\Gamma)$ if $\delta=1$.
\end{condition}

\begin{remark}
Note that Condition \ref{condcont} implies Condition \ref{conddensity} and Condition \ref{condhamza}. In particular, Condition \ref{condcont} holds if $\alpha, \beta \in C^1(\overline{\Omega})$, $\alpha >0$ $\lambda$-a.e. on $\Omega$ and $\beta >0$ $\sigma$-a.e. on $\Gamma$.
\end{remark}

\begin{proposition} \label{propgen}
Suppose that Condition \ref{condcont} is satisfied and let $f \in C^2(\overline{\Omega})$. Then 
\[ Lf= \frac{1}{2} \Big( \mathbbm{1}_{\Omega}~ \big( \Delta f + (\nabla \ln \alpha, \nabla f) \big) - \mathbbm{1}_{\Gamma} ~ \frac{\alpha}{\beta}~(n, \nabla f) + \delta~ \mathbbm{1}_{\Gamma}~ \big( \Delta_{\Gamma} f + (\nabla_{\Gamma} \ln \beta, \nabla_{\Gamma} f) \Big). \]
\end{proposition}

\begin{proof}
Let $f \in C^2(\overline{\Omega})$ and $g \in \mathcal{D}=C^1(\overline{\Omega})$. Then we get by the divergence theorem on $\Omega$ and (\ref{divergence}):
\begin{align*} 
&\mathcal{E}(f,g) \\
&=\frac{1}{2} \int_{\Omega} (\nabla f, \nabla g)~ \alpha d\lambda + \frac{\delta}{2} \int_{\Gamma} (\nabla_{\Gamma} f,\nabla_{\Gamma} g) ~\beta d\sigma \\
&= \frac{1}{2} \int_{\Omega} (\alpha \nabla f, \nabla g) ~d\lambda + \frac{\delta}{2} \int_{\Gamma} (\beta \nabla_{\Gamma} f, \nabla_{\Gamma} g) ~ d\sigma \\
&= - \frac{1}{2} \int_{\Omega} g~ \text{div} (\alpha \nabla f) ~d\lambda + \frac{1}{2} \int_{\Gamma} g~ (\nabla f, n) ~\alpha d\sigma - \frac{\delta}{2} \int_{\Gamma} g~ \text{div}_{\Gamma}(\beta \nabla_{\Gamma} f)~ d\sigma \\
&=- \frac{1}{2} \int_{\Omega} g~ (\Delta f \alpha + (\nabla \alpha, \nabla f)) ~d\lambda + \frac{1}{2} \int_{\Gamma} g~ \frac{\alpha}{\beta}(\nabla f, n) ~\beta d\sigma - \frac{\delta}{2} \int_{\Gamma} g (\beta \Delta_{\Gamma} f + (\nabla_{\Gamma} \beta, \nabla_{\Gamma} f))~d\sigma
\end{align*}
Note that $\nabla \ln \alpha= \frac{\nabla \alpha}{\alpha}$ on $\{ \alpha >0\}$ and $\nabla_{\Gamma} \ln \beta= \frac{\nabla_{\Gamma} \beta}{\beta}$ on $\{ \beta >0\}$. In other words, ${\nabla \alpha= \nabla \ln \alpha ~ \alpha}$ and $\nabla_{\Gamma} \beta= \nabla_{\Gamma} \ln \beta ~ \beta$. Hence, we get 
\[ \mathcal{E}(f,g)=(-Lf,g)_{L^2(\overline{\Omega};\mu)} \]
for $f \in C^2(\overline{\Omega})$ and $g \in \mathcal{D}$. By density of $\mathcal{D}$ in $D(\mathcal{E})$ with respect to the $\mathcal{E}_1$-norm, the claim follows.
\end{proof}

We can define the operator $L_{\Omega}$ and the boundary operator $L_{\Gamma}$ by 
\[ L_{\Omega} f:=\frac{1}{2} \big( \Delta f + ( \nabla \ln \alpha, \nabla f)  \big) ~\text{ and } ~L_{\Gamma} f:= \frac{1}{2} \big( \delta~ \Delta_{\Gamma} f + \delta~ ( \nabla_{\Gamma} \ln \beta, \nabla_{\Gamma} f) - \frac{\alpha}{\beta} (n, \nabla f)  \big) \]
for $f \in C^2(\overline{\Omega})$. Then the generator $L$ has the representation $L f= \mathbbm{1}_{\Omega}~ L_{\Omega} f + \mathbbm{1}_{\Gamma}~ L_{\Gamma} f$. The associated Cauchy problem for $g \in C^2(\overline{\Omega})$ has the form
\begin{align} \label{wentzellcond}
\left\{
\begin{array}{l l}
 \frac{\partial}{\partial t} u_t = \frac{1}{2} \big( \Delta u_t + ( \nabla \ln \alpha, \nabla u_t)  \big), & \text{on } \overline{\Omega}, t >0  \\
 \Delta u_t + ( \nabla \ln \alpha, \nabla u_t)  - \delta~ \Delta_{\Gamma} u_t - \delta~ ( \nabla_{\Gamma} \ln \beta, \nabla_{\Gamma} u_t) + \frac{\alpha}{\beta} (n, \nabla u_t)  =0, & \text{on } \Gamma, t >0, \\
 u_{\scriptscriptstyle{0}}=g & \text{on } \overline{\Omega}.
\end{array}
\right.
\end{align}
The condition in (\ref{wentzellcond}) is called Wentzell boundary condition. Note that if we multiply (\ref{wentzellcond}) by $\beta$ and then set $\beta$ to zero, the equation reduces to the Neumann boundary condition.\\

For $h \in C^1(\overline{\Omega})$, we have by definition and calculation $(\nabla_{\Gamma} h, \nabla_{\Gamma} f)= (P \nabla h, \nabla f)$ and \linebreak $\Delta_{\Gamma} f = \text{Tr} (P \nabla^2 f) - (n, \nabla f)~ \text{Tr} ( P \nabla n)=\text{Tr} (P \nabla^2 f) - (\kappa n , \nabla f)$. Hence, we get with 
\begin{align} \label{defA} A:= \mathbbm{1}_{\Omega} E + \delta~ \mathbbm{1}_{\Gamma} P 
\end{align}
as well as
\begin{align} \label{defb} b:= \frac{1}{2} \Big( \mathbbm{1}_{\Omega} \nabla \ln \alpha + \mathbbm{1}_{\Gamma} \big( \delta~ P \nabla \ln \beta - (\frac{\alpha}{\beta}+ \kappa) n \big) \Big) 
\end{align}
the representation
\begin{align} \label{generator}
Lf=\frac{1}{2} \text{Tr}( A \nabla^2 f) + (b,\nabla f). 
\end{align}
Note that $AA^t=A^2=A$.

\subsection{Solution to the martingale problem and SDE} \label{secSDE}

\begin{theorem}
The diffusion process $\mathbf{M}$ from Theorem \ref{thmdiff} is up to $\mu$-equivalence the unique diffusion process having $\mu$ as symmetrizing measure and solving the martingale problem for $(L,D(L))$, i.e., for all $g \in D(L)$
\[ \tilde{g}(\mathbf{X}_t) - \tilde{g}(\mathbf{X}_0) - \int_0^t (Lg)(\mathbf{X}_s) ds, \ t \geq 0, \]
is an $\mathcal{F}_t$-martingale under $\mathbf{P}_x$ for quasi all $x \in \overline{\Omega}$. Here $\tilde{g}$ denotes a quasi-continuous version of $g$ (for the definition of quasi-continuity see e.g. \cite[Chap. IV, Proposition 3.3]{MR92}).
\end{theorem}

\begin{proof}
See e.g. \cite[Theorem 3.4 (i)]{AR95}.
\end{proof}

By the explicit calculation of $L$ given in Proposition \ref{propgen} and the notation in (\ref{generator}), we obtain the following corollary:

\begin{corollary}
Assume that Condition \ref{condcont} is fulfilled. Let $g \in C^2(\overline{\Omega})$ and let $\mathbf{M}$ be the diffusion process from Theorem \ref{thmdiff}. Then 
\[ g(\mathbf{X}_t) - g(\mathbf{X}_0) - \int_0^t \frac{1}{2} \text{Tr} \big( A(\mathbf{X}_s) \nabla^2 g(\mathbf{X}_s) \big) + \big( b(\mathbf{X}_s), \nabla g(\mathbf{X}_s) \big) ds, \ t \geq 0, \]
is an $\mathcal{F}_t$-martingale under $\mathbf{P}_x$ for quasi every $x \in \overline{\Omega}$, where $A$ and $b$ are defined as in (\ref{defA}) and (\ref{defb}).
\end{corollary}

\begin{lemma}[weak solutions and martingale problems] \label{lemSV}
Fix the probability measure $\mathbf{P}=\mathbf{P}_x$, $x \in \overline{\Omega}$, on $C(\mathbb{R}_{+},\mathbb{R}^d)$ (see also Remark \ref{remcanonical}). Let $A$, $b$ be given on $\overline{\Omega}$ by (\ref{defA}) and (\ref{defb}) respectively. If 
\[ f(\mathbf{X}_t) - f(\mathbf{X}_0) - \int_0^t \frac{1}{2} \text{Tr} \big( A(\mathbf{X}_s) \nabla^2 f (\mathbf{X_s}) \big) + (b(\mathbf{X}_s), \nabla f(\mathbf{X}_s)) \ ds \]
is an $\mathcal{F}_t$-martingale under $\mathbf{P}$ for every $f \in C^{\infty}_c(\mathbb{R}^d)$, the equation
\[ d\mathbf{X}_t = A(\mathbf{X}_t) dB_t + b(\mathbf{X}_t) dt \]
has a weak solution with distribution $\mathbf{P}$, where $(B_t)_{t \geq 0}$ is a $d$-dimensional standard Brownian motion.
\end{lemma}

\begin{proof}
See e.g. \cite[Theorem 18.7]{Kal97}. Note that $A$ and $b$ (defined as in (\ref{defA}) and (\ref{defb})) fulfill the required conditions, since they are progressive. Furthermore, the property $AA^t=A$ is used.
\end{proof}

\begin{remark}
The solution to the SDE given in Lemma \ref{lemSV} results from $\mathbf{M}$ by extending the underlying filtration $(\mathcal{F}_t)_{t \geq 0}$ if necessary (see proof of  \cite[Theorem 18.7]{Kal97} and the references therein). For convenience, we use for the process equipped with the enlarged filtration again the notation
\[ \mathbf{M}=\big( \mathbf{\Omega}, \mathcal{F}, (\mathcal{F}_t)_{t \geq 0}, (\mathbf{X}_t)_{t \geq 0}, (\Theta_t)_{t \geq 0}, (\mathbf{P}_x)_{x \in \overline{\Omega}} \big)  \]
taking into account that the associated Dirichlet form is still given by $(\mathcal{E},D(\mathcal{E}))$.
\end{remark}


\begin{theorem} \label{thmsolSDE}
$\mathbf{M}$ is a solution to the SDE
\begin{align*}
d\mathbf{X}_t =& \mathbbm{1}_{\Omega} (\mathbf{X}_t) \Big( dB_t + \frac{1}{2} \nabla \ln \alpha (\mathbf{X}_t) dt \Big) - \mathbbm{1}_{\Gamma}(\mathbf{X}_t) \frac{\alpha}{\beta}(\mathbf{X}_t) ~n(\mathbf{X_t}) dt   \\ 
&+ \delta ~ \mathbbm{1}_{\Gamma}(\mathbf{X}_t) \Big( dB_t^{\Gamma} + \frac{1}{2} \nabla_{\Gamma} \ln \beta (\mathbf{X}_t) dt \Big),  \\
dB_t^{\Gamma} =& P(\mathbf{X}_t) \circ dB_t,  \\
\mathbf{X}_0 =& x, 
\end{align*}
for quasi every starting point $x \in \overline{\Omega}$, where $(B_t)_{t \geq 0}$ is a $d$-dimensional standard Brownian motion, i.e.,
\begin{align}
\mathbf{X}_t = x &+ \int_0^t \mathbbm{1}_{\Omega}(\mathbf{X}_s) dB_s + \int_0^t \mathbbm{1}_{\Omega}(\mathbf{X}_s) \frac{1}{2} \nabla \ln \alpha(\mathbf{X}_s) ds \notag \\
&+ \delta \int_0^t \mathbbm{1}_{\Gamma}(\mathbf{X}_s) P(\mathbf{X}_s) dB_s - \delta \int_0^t  \mathbbm{1}_{\Gamma}(\mathbf{X}_s) \frac{1}{2} \kappa(\mathbf{X}_s) n(\mathbf{X}_s) ds \label{qesolution} \\
&+ \delta \int_0^t  \mathbbm{1}_{\Gamma}(\mathbf{X}_s) \frac{1}{2} \nabla_{\Gamma} \ln \beta (\mathbf{X}_s) ds - \int_0^t \frac{1}{2} \frac{\alpha}{\beta}(\mathbf{X}_s)  \mathbbm{1}_{\Gamma}(\mathbf{X}_s)n(\mathbf{X}_s) ds \notag
\end{align}
almost surely under $\mathbf{P}_x$ for quasi every $x \in \overline{\Omega}$.
\end{theorem}

\begin{remark}
A Fukushima decomposition of $\mathbf{M}$ (see \cite[Chap. 5]{FOT94}) yields the same result as in Theorem \ref{thmsolSDE}. We would like to mention that the argument used here in order to get a solution to the SDE (\ref{SDE}) does not work for reflecting (Neumann) boundary conditions, since in this case the reflection is not given by a drift term. However, a Fukushima decomposition is still valid (see e.g. \cite{Tru03}), because in this case it is also possible to assign an additive functional to the surface measure $\sigma$. The advantage in our situation is that we are able to express the boundary behavior in terms of the generator.
\end{remark}

\begin{remark}
We define the boundary local time $(l_t)_{t \geq 0}$ of $(\mathbf{X}_t)_{t\geq 0}$ as the additive functional corresponding to the measure $\sigma$ on $\Gamma$ (in the sense of Revuz correspondence). Then, $(l_t)_{t\geq 0}$ is given by
\[ l_t=\int_0^t \frac{1}{\beta}(\mathbf{X}_s)~ \mathbbm{1}_{\Gamma}(\mathbf{X}_s)~n(\mathbf{X}_s) ds \]
and in view of (\ref{qesolution}) $(\mathbf{X}_t)_{t \geq 0}$ has the representation
\begin{align*}
\mathbf{X}_t = x &+ \int_0^t \mathbbm{1}_{\Omega}(\mathbf{X}_s) dB_s + \int_0^t \mathbbm{1}_{\Omega}(\mathbf{X}_s) \frac{1}{2} \nabla \ln \alpha(\mathbf{X}_s) ds  \\
&+ \delta \int_0^t \mathbbm{1}_{\Gamma}(\mathbf{X}_s) P(\mathbf{X}_s) dB_s - \delta \int_0^t  \mathbbm{1}_{\Gamma}(\mathbf{X}_s) \frac{1}{2} \kappa(\mathbf{X}_s) n(\mathbf{X}_s) ds  \\
&+ \delta \int_0^t  \mathbbm{1}_{\Gamma}(\mathbf{X}_s) \frac{1}{2} \nabla_{\Gamma} \ln \beta (\mathbf{X}_s) ds - \int_0^t \frac{1}{2} \alpha(\mathbf{X}_s) dl_s
\end{align*}
almost surely under $\mathbf{P}_x$ for quasi every $x \in \overline{\Omega}$. \\
Consider the Dirichlet form given by the closure of
\begin{align}
\frac{1}{2} \int_{\Omega} (\nabla f, \nabla g)~ \alpha d\lambda + \frac{\delta}{2} \int_{\Gamma} (\nabla_{\Gamma} f,\nabla_{\Gamma} g) ~\beta d\sigma \ \text{ for } f,g \in \mathcal{D}:=C^1(\overline{\Omega}), \label{reflDF} \end{align}
on $L^2(\Omega;\alpha \lambda)$, $\delta \in \{0,1\}$. For $\delta=0$ this construction yields reflecting distorted Brownian motion on $\overline{\Omega}$ and for $\delta=1$ the setting corresponds to a generalization of the one considered in \cite{Car09}. More precisely, the associated process is a solution to
\begin{align}
\mathbf{X}_t = x &+ B_t + \int_0^t \frac{1}{2} \nabla \ln \alpha(\mathbf{X}_s) ds  \notag \\
&+ \delta \int_0^t \beta(\mathbf{X}_s)~ P(\mathbf{X}_s) dB_{L_s} - \delta \int_0^t \beta(\mathbf{X}_s)~ \frac{1}{2} \kappa(\mathbf{X}_s) n(\mathbf{X}_s) dL_s \label{reflSDE} \\
&+ \delta \int_0^t \frac{1}{2} \nabla_{\Gamma} \beta (\mathbf{X}_s) dL_s - \int_0^t \frac{1}{2} \alpha(\mathbf{X}_s) n(\mathbf{X}_s) dL_s, \notag
\end{align}
where $(L_t)_{t \geq 0}$ is the local time of the diffusion associated to the closure of (\ref{reflDF}) on $L^2(\Omega;\alpha \lambda)$ (in the sense of the additive functional associated to $\sigma$). 
Using the connection of random time changes and Dirichlet forms presented in \cite[Chapter 5]{CF11} and \cite[Chapter 6]{FOT94}, it follows that the change of the reference measure from $\alpha \lambda$ to $\alpha \lambda+ \beta \sigma$ corresponds to a random time change of the associated process via the inverse $\tau(t)$ of the additive functional given by
\[ A_t:=t + \beta(\mathbf{X}_t) L_t, \quad t \geq 0. \]
In other words, the sticky reflecting diffusion (\ref{qesolution}) results from a diffusion with ordinary reflection at $\Gamma$ by introducing a new time scale such that time slows down if (and only if) the process is located on the boundary of $\Omega$. Note that it is reasonable that the new time scale $\tau(t)$ converts a solution of (\ref{reflSDE}) into a solution of (\ref{qesolution}). This connection has already been observed in the case of sticky reflected distorted Brownian motion on $[0, \infty)^n$, $n\in \mathbb{N}$, in \cite{GV14}.
\end{remark}

\subsection{Ergodicity and occupation time} 

Throughout this section we assume that Condition \ref{conddensity} and Condition \ref{condhamza} are fulfilled and denote by $\mathbf{M}$ the process constructed in Theorem \ref{thmdiff}. Given the process $\mathbf{M}$, we can define via its transition semigroup $(p_t)_{t \geq 0}$ a Dirichlet form and by construction of $\mathbf{M}$ this form is $(\mathcal{E},D(\mathcal{E}))$ again. Recall that the sub-Markovian strongly continuous contraction semigroup of $(\mathcal{E},D(\mathcal{E}))$ is denoted by $(T_t)_{t \geq 0}$. We use the results provided in \cite[Chap. 4.7]{FOT94} in order to prove an ergodic theorem for $\mathbf{M}$. To do this, we restrict to invariant subsets of $\overline{\Omega}$ and show the part of the form $(\mathcal{E},D(\mathcal{E}))$ on the invariant set is irreducible recurrent. This allows to determine the occupation time of the process on $\Gamma$ and, as a consequence, to show that the boundary behavior is indeed sticky. The main result of this section is Theorem \ref{thmergo}. In order to avoid confusion, we label the capacity of a set by the underlying Dirichlet form. For the sake of convenience, we state all proofs for the case $\delta=1$, which can easily be modified to hold for $\delta=0$. \\

First, we define the notion of parts of Dirichlet forms:

\begin{definition}[\emph{part of a Dirichlet form}] \label{defpart} Let $(\mathcal{G}, D(\mathcal{G}))$ be an arbitrary regular Dirichlet form on some locally compact, separable metric space $X$, $m$ a positive Radon measure on $X$ with full topological support and $G$ an open subset of $X$. Then we define by $\mathcal{G}^G(f,g):=\mathcal{G}(f,g)$ for $f,g \in D(\mathcal{G}^G):= \{ f \in D(\mathcal{G}) | \ \tilde{f}=0 \ \mathcal{G}\text{-q.e. on } X \backslash G \}$ the \emph{part of the form $(\mathcal{G},D(\mathcal{G}))$ on $G$}, where $\tilde{f}$ denotes an $\mathcal{G}$-quasi-continuous version of $f$. Indeed, this defines a regular Dirichlet form on $L^2(G; m)$ (see \cite[Theorem 4.4.3]{FOT94}).
\end{definition}

Throughout this section, suppose that Condition \ref{condcont} is satisfied and denote by 
\begin{align*}
\mathbf{M}:=\big( \mathbf{\Omega}, \mathcal{F}, (\mathcal{F}_t)_{t \geq 0}, (\mathbf{X}_t)_{t \geq 0}, (\Theta_t)_{t \geq 0}, (\mathbf{P}_x)_{x \in \overline{\Omega}} \big)
\end{align*}
the process constructed in Section \ref{secSDE}. Furthermore, for an open subset $G$ of $\overline{\Omega}$
\begin{align*} 
\mathbf{M}^G:=\big( \mathbf{\Omega}, \mathcal{F}, (\mathcal{F}_t)_{t \geq 0}, (\mathbf{X}^0_t)_{t \geq 0}, (\Theta_t)_{t \geq 0}, (\mathbf{P}_x)_{x \in G_{\Delta}} \big)
\end{align*}
is called the part of the process $\mathbf{M}$ on $G$, where $\mathbf{X}^0_t(\omega)$ results from $\mathbf{X}_t(\omega)$ by killing the path upon leaving $G$ for $\omega \in \mathbf{\Omega}$. By \cite[Theorem 4.4.2]{FOT94} the process $\mathbf{M}^G$ is associated to $(\mathcal{E}^G,D(\mathcal{E}^G))$. \\

Let $\mathcal{C}$ be the set of all connected components of $\overline{\Omega}_1:=\overline{\Omega} \backslash \{\varrho=0\} = \{ \varrho >0\}$ and for $G \in \mathcal{C}$ let $G_{\Gamma}:= G \cap \Gamma$.

\begin{condition} \label{condcomp} $\text{cap}_{\mathcal{E}} (\{ \varrho=0\})=0$ and $\alpha, \beta \in C(\overline{\Omega})$.
\end{condition}

\begin{lemma} Assume that Condition \ref{condcomp} is fulfillded. Then
\begin{enumerate}
\item $\text{cap}_{\mathcal{E}_{\Omega}} (\{ \varrho =0\})=0$ and $\text{cap}_{\mathcal{E}_{\Gamma}} (\{ \varrho =0\} \cap \Gamma)=0.$
\item \label{invcomp} Each $G \in \mathcal{C}$ is open in $\overline{\Omega}$ and quasi closed with respect to $\mathcal{E}$. In particular, G is $T_t$-invariant.
\item The assertion in \ref{invcomp} holds accordingly for $G$ and $G_{\Gamma}$ with respect to $\mathcal{E}_{\Omega}$ and $\mathcal{E}_{\Gamma}$ respectively.
\end{enumerate}
\end{lemma}

\begin{proof}
(i) Note that $D(\mathcal{E})$ is a subset of $D(\mathcal{E}_{\Omega})$ and $D(\mathcal{E}_{\Gamma})$ by restriction and $\mathcal{E}_{\Omega,1}, \mathcal{E}_{\Gamma,1} \leq \mathcal{E}_1$ on this set. Let $\varepsilon  >0$. Then there exists an open set $U$ in $\overline{\Omega}$ which contains $\{ \varrho=0 \}$ such that $\text{cap}_{\mathcal{E}} (U) < \varepsilon$. By definition of the capacity, we get also 
$\text{cap}_{\mathcal{E}_{\Omega}}(U) < \varepsilon$ and $\text{cap}_{\mathcal{E}_{\Gamma}}(U \cap \Gamma) < \varepsilon$. Hence, the assertion holds true.\\
(ii) Let $G \in \mathcal{C}$. $G$ is open by definition. Let $\varepsilon >0$. We show that $\overline{\Omega} \backslash G$ is quasi open. Since $\text{cap}_{\mathcal{E}} (\{ \varrho=0\})=0$, there exists an open subset $U$ of $\overline{\Omega}$ which contains $\{ \varrho=0 \}$ such that $\text{cap}_{\mathcal{E}} (U) < \varepsilon$. The set $\tilde{U}:=\bigcup_{\tilde{G} \in \mathcal{C} \backslash \{ G\}} \tilde{G} \cup~ U$ is open and contains $\overline{\Omega} \backslash G$. Moreover, 
\[ \text{cap}_{\mathcal{E}}(\tilde{U} \backslash (\overline{\Omega} \backslash G)) \leq \text{cap}_{\mathcal{E}}(U) < \varepsilon. \]
Hence, $G$ is quasi closed.\\
(iii) Note that $G$ and $G_{\Gamma}$ are open in $\overline{\Omega}$ and $\Gamma$ respectively. The remaining part of the statement follows by (i) with the same arguments as in (ii).
\end{proof}

\begin{remark} \label{reminv}
\begin{enumerate} Let $G \in \mathcal{C}$.
\item Due to \cite[Lemma 4.6.3]{FOT94}, the preceding lemma implies that there exists a properly exceptional set $N$ such that $G \backslash N$ is $\mathbf{M}$-invariant in the sense that 
\[ \mathbf{P}_x(\mathbf{X}_t \in (G \backslash N)_{\Delta} \text{ for all } t \geq 0)=1 \text{ for all } x \in G \backslash N. \]
\item It is possible that $G_{\Gamma}=G \cap \Gamma$ is not connected in $\Gamma \backslash \{ \varrho=0\}$. Therefore, we denote by $\mathcal{C}_G$ the set containing all connected components of $G_{\Gamma}$. In particular, $\bigcup_{G \in \mathcal{C}} \mathcal{C}_G$ is the set of all connected components of $\Gamma \backslash \{ \varrho=0\}$.
\item Define $F_k :=\{ x \in G | \ d_{\text{euc}}(x, \{ \rho =0\}) > \frac{1}{k} \}$. This yields a sequence of open subsets of $G$ increasing to $G$. For $\alpha, \beta \in C(\overline{\Omega})$, it follows that $\gamma_k := \text{ess inf}_{x \in F_k} \ \varrho >0$, $k=1,2,\dots,$ (with respect to the measure $\lambda$). Similarly, we define $F_k$ for sets in $\mathcal{C}_G$. More precisely, for $A_G \in \mathcal{C}_G$ let $F_k :=\{ x \in A_G | \ d_{\text{euc}}(x, \{ \rho =0\} \cap \Gamma) > \frac{1}{k} \}$ and define $\gamma_k$ with respect to $\sigma$.
\item By a similar argument as in (iii), $L^p$-norms on $K$ with respect to the measures $\mu$ and $\lambda$ (or $\sigma$) respectively are equivalent for some compact set $K$ properly contained in some $G$ (or $A_G$).  
\end{enumerate}
\end{remark}

\begin{theorem} \label{thmergo}
Suppose that Condition \ref{condcomp} is fulfilled. Then 
for all $G \in \mathcal{C}$ and $f \in L^1(G;\mu)$ holds
\[ \lim_{t \rightarrow \infty} \frac{1}{t} \int_0^t f(\mathbf{X}_s) ds = \frac{\int_G f d\mu}{\mu(G)} \]
alomost surely under $\mathbf{P}_x$ for quasi all $x \in G$.
\end{theorem}

\begin{proof}
Fix $G \in \mathcal{C}$. Due to \cite[Theorem 4.7.3(iii)]{FOT94}, the definition of $\mathbf{M}^G$ and Remark \ref{reminv} (i) it is sufficient to show that $(\mathcal{E}^G,D(\mathcal{E}^G))$ is irreducible recurrent. In order to deduce recurrence of $(\mathcal{E},D(\mathcal{E}))$, by \cite[Theorem 1.6.3]{FOT94} it is enough to observe that $1_{\overline{\Omega}} \in D(\mathcal{E})$ and $\mathcal{E}(1_{\overline{\Omega}} ,1_{\overline{\Omega}})=0$.  Hence, $\mathbbm{1}_G = \mathbbm{1}_G \mathbbm{1}_{\overline{\Omega}} \in D(\mathcal{E}^G)$ by $T_t$-invariance of $G$ and $\mathcal{E}^G(\mathbbm{1}_G,\mathbbm{1}_G)=0$, since
\[ 0=\mathcal{E}(\mathbbm{1}_{\overline{\Omega}},\mathbbm{1}_{\overline{\Omega}}) = \mathcal{E}(\mathbbm{1}_G,\mathbbm{1}_G) + \mathcal{E}(\mathbbm{1}_{\overline{\Omega} \backslash G},\mathbbm{1}_{\overline{\Omega} \backslash G}). \]
This implies recurrence of $(\mathcal{E}^G,D(\mathcal{E}^G))$ by \cite[Theorem 1.6.3]{FOT94}. Taking into account that the considered form is recurrent, irreducibility is equivalent to the condition that every $f \in D(\mathcal{E}^G)$ with $\mathcal{E}^G(f,f)=0$ is $\mu$-a.e. constant (on $G$) by \cite[Theorem 2.1.11]{CF11}. Let $A_G \in \mathcal{C}_G$ and denote by $(\mathcal{E}^{A_G}_{\Gamma},D(\mathcal{E}^{A_G}_{\Gamma}))$ the part of the form $(\mathcal{E}_\Gamma,D(\mathcal{E}_\Gamma))$ on $A_G$. Moreover, denote by $(\mathcal{E}^{G}_{\Omega},D(\mathcal{E}^{G}_{\Omega}))$ the part of the form $(\mathcal{E}_\Omega,D(\mathcal{E}_\Omega))$ on $G$. $(\mathcal{E}^{A_G}_{\Gamma},D(\mathcal{E}^{A_G}_{\Gamma}))$ is the closure of $(\mathcal{E}_{\Gamma},C^1(A_G))$ by \cite[Theorem 4.4.3]{FOT94} and thus, it is irreducible. Indeed, the closure of the pre-Dirichlet form
\[ \int_{A_G} (\nabla_{\Gamma} f, \nabla_{\Gamma} g)~ d\sigma, \ f,g \in C^1(\overline {A_G}) \]
on $L^2(A_G;\sigma)$ yields reflecting Brownian motion which is irreducible (see e.g.  \cite[p.128]{CF11}). Hence, the closure of the form defined for functions in $C^1(A_G)$ on $L^2(A_G;\sigma)$ is also irreducible in view of \cite[Theorem 2.1.11]{CF11}. Hence, it follows by \cite[Corollary 4.6.4]{FOT94} and Remark \ref{reminv} (iii) that $(\mathcal{E}_{\Gamma}^{A_G},D(\mathcal{E}_{\Gamma}^{A_G}))$ is irreducible. Similarly, it holds that  $(\mathcal{E}_{\Omega}^G,D(\mathcal{E}_{\Omega}^G))$ is irreducible.\\
Let $f \in D(\mathcal{E}^G)$ and choose a seqeunce $(f_k)_{k \in \mathbb{N}}$ in $C^1(G)$ such that $f_k \rightarrow f$ with respect to $\sqrt{\mathcal{E}^G_1}$.
Then the restriction to $\Gamma$ is by definition $\mathcal{E}_{\Gamma}$-Cauchy and converges to the restriction of $f$ in $L^2(\Gamma;\beta \sigma)$. Therefore, the convergence holds also in $D(\mathcal{E}_{\Gamma})$. An analogous statement holds in $D(\mathcal{E}_{\Omega})$. Thus,
\[ \mathcal{E}^G(f,f)=\mathcal{E}(f,f)= \lim_{k \rightarrow \infty} \mathcal{E}(f_k,f_k)=\lim_{k \rightarrow \infty} (\mathcal{E}_{\Omega}(f_k,f_k)+ \mathcal{E}_{\Gamma}(f_k,f_k)) =\mathcal{E}_{\Omega}(f,f)+ \mathcal{E}_{\Gamma}(f,f) \]
by definition. By invariance it holds
\[ \mathcal{E}^G(f,f) =  \mathcal{E}^{G}_{\Omega}(\mathbbm{1}_{G} f, \mathbbm{1}_{G} f) + \sum_{A_G \in \mathcal{C}_G} \mathcal{E}_{\Gamma}^{A_G}(\mathbbm{1}_{A_G} f,\mathbbm{1}_{A_G} f).\]
Therefore, $\mathcal{E}^G(f,f)=0$ implies that each summand on the right hand side vanishes and hence, $f=c_G$ $\alpha \lambda$-a.e. on $G \cap \Omega$ for some constant $c_G$ and $f=c_{A_G}$ $\beta \sigma$-a.e. on $A_G$ for some constant $c_{A_G}$ by \cite[Theorem 2.1.11]{CF11} and irreducibility. Thus, we can conclude \[ f = c_G \mathbbm{1}_{G \cap \Omega} + \sum_{A_G \in \mathcal{C}_G} c_{A_G} \mathbbm{1}_{A_G}.\]
It rests to show that $c_G=c_{A_G}$ for all $A_G \in \mathcal{C}_G$. Fix a point $z \in A_G$. Then there exists a neighborhood $U$ of $z$ in $\overline{\Omega}$ such that $\overline{U} \subset G$ and $\overline{U} \cap \Gamma \subset A_G$. Choose a $C^{\infty}$-cutoff function $\eta$ defined on $\overline{\Omega}$ which is constantly one near $z$  and has support properly contained in $U$. Then it is easy to see that $\eta f \in D(\mathcal{E}^G)$ and $(\eta f_k)_{k \in \mathbb{N}}$ is an approximation for $\eta f$ whenever $(f_k)_{k \in \mathbb{N}}$ is a sequence of $C^1(G)$-functions which approximates $f$ in $D(\mathcal{E}^G)$. In particular, this implies convergence in $L^2(U \cap \Gamma;\sigma)$ and even in $L^2(\partial(U \cap \Omega);\sigma)$. Since $\eta c_G$ is the unique continuous extension of $f|_{U \cap \Omega}$ to $U$, it is clear that $\eta f \in H^{1,2}(U \cap \Omega) \cap C(U \cap \Omega)$ and $\text{Tr}(\eta f)=\eta c_G$, where $\text{Tr}: H^{1,2}(U \cap \Omega) \rightarrow L^2(U \cap \Gamma;\sigma)$ is the (restricted) trace operator. Thus,
\[ \eta c_{G} = \text{Tr}(\eta f) = L^2( U \cap \Gamma;\sigma)-\lim_{k \rightarrow \infty} \text{Tr}(\eta f_k) = L^2( U \cap \Gamma;\sigma)-\lim_{k \rightarrow \infty} (\eta f_k)|_{U \cap \Gamma} = \eta c_{A_G}. \]
Hence, $\eta c_{A_G} = \eta c_G$ $\sigma$-a.e. on $U \cap \Gamma$ and therefore, $c_{A_G}=c_G$.
\end{proof}

\begin{corollary} \label{coroergo}
Suppose that Condition \ref{condcomp} is fulfilled. Fix a component $G$ of $\overline{\Omega}_1$ which intersects $\Gamma$. Then
\begin{align} \label{occtime} \lim_{t \rightarrow \infty} \frac{1}{t} \int_0^t \mathbbm{1}_{\Gamma}(\mathbf{X}_s) ds = \frac{\mu(G \cap \Gamma)}{\mu(\overline{\Omega})}
\end{align}
alomost surely under $\mathbf{P}_x$ for quasi all $x \in G$.
\end{corollary}

\begin{proof}
Under the above assumptions, we have $\mathbbm{1}_{\Gamma}(\mathbf{X}_s)=\mathbbm{1}_{\Gamma \cap G}(\mathbf{X}_s)$ for all $s \geq 0$ $\mathbf{P}_x$-a.s. for quasi every ${x \in G}$. Hence, the assertion follows by Theorem \ref{thmergo} with $f=\mathbbm{1}_{G \cap \Gamma} \in L^1(G;\mu)$.
\end{proof}

\begin{remark}
Note that the right hand side of (\ref{occtime}) is strictly positive if $\mu(G \cap \Gamma) >0$ and there exists always some $G \in \mathcal{C}$ such that $\mu(G \cap \Gamma)>0$, since $\mu(\Gamma)>0$. This implies that the process sojourns arbitrarily long on $\Gamma$.
\end{remark}

For the following theorem we need the notion of a strongly regular Dirichlet form (see also \cite{Stu94} and \cite{Stu95}):

\begin{definition}[\emph{strong regularity}] A regular Dirichlet form $(\mathcal{G},D(\mathcal{G}))$ on $L^2(X;m)$, where $X$ is a connected, locally compact, separable Hausdorff space and $m$ is a positive Radon measure with full support, is called \emph{strongly regular}, if the topology induced by the intrinsic metric
\[ d(x,y):= \sup \{ f(x)-f(y) | \ f \in D(\mathcal{G}) \cap C(X) \text{ with } \nu_{\scriptscriptstyle{\langle f\rangle}} \leq m \}, \ \ x,y \in X, \]
coincides with the original topology on $X$. Here $\nu_{\scriptscriptstyle{\langle f\rangle}} \leq m$ means that the so-called energy measure of $f$ is absolutely continuous with respect to $m$ and its Radon-Nikodym derivative $\frac{d \nu_{\langle f \rangle}}{d m}$ is almost everywhere less or equal than one.
\end{definition}

\begin{lemma} \label{lemweakdiff}
Let $\emptyset \neq U$ be an open subset of $\overline{\Omega} \backslash \{ \varrho=0\}$ such that $\overline{U} \subset \overline{\Omega} \backslash \{ \varrho=0\}$. Then the restriction maps $i_1: f \mapsto f|_{U \cap \Omega}$ and $i_2: f \mapsto f|_{U \cap \Gamma}$ (under the condition that $U \cap \Gamma \neq \emptyset$) are continuous maps from $D(\mathcal{E})$ to $H^{1,2}(U \cap \Omega)$ and $H^{1,2}(U \cap \Gamma)$ respectively. In particular, there exists a constant $C_1=C_1(\varrho) < \infty$ such that $\Vert f \Vert_{H^{1,2}(U \cap \Omega)} , \Vert f \Vert_{H^{1,2}(U \cap \Gamma)} \leq C_1 \sqrt{\mathcal{E}_1(f,f)}$ for $f \in D(\mathcal{E})$.
\end{lemma}

\begin{proof}
By continuity of $\alpha$ and $\beta$, there exist constants $0 < \varrho^{-}$ and $\varrho^{+} < \infty$ such that $\varrho^{-} \leq \varrho \leq \varrho^{+}$ on $\overline{U}$. Let $f \in \mathcal{D}$. Then 
\begin{align*} \int_{U \cap \Omega} (f^2 + |\nabla f|^2) ~d\lambda &\leq \frac{1}{\varrho^{-}}  \int_{U \cap \Omega} (f^2 + |\nabla f|^2)~ \alpha d\lambda \\
&\leq \frac{1}{\varrho^{-}} \int_{\Omega} (f^2 + |\nabla f|^2)~ \alpha d\lambda \\
&\leq \frac{1}{\varrho^{-}} \mathcal{E}_1(f,f) < \infty.
\end{align*}
Similary, we obtain 
\[ \int_{U \cap \Gamma} (f^2 + |\nabla_{\Gamma} f|^2) ~d\sigma \leq \frac{1}{\varrho^{-}} \mathcal{E}_1(f,f) < \infty. \]
Hence, $i_1: \mathcal{D} \rightarrow H^{1,2}(U \cap \Omega)$ and $i_1: \mathcal{D} \rightarrow H^{1,2}(U \cap \Omega)$ are well-defined and continuous. Therefore, the maps admit a continuous extension to $D(\mathcal{E})$. Let $f \in D(\mathcal{E})$. Then the image of $f$ is simply the restriction of $f$ to the respective set (see also Remark \ref{reminv} (iv)) and thefore, the restriction is an element of the corresponding Sobolev space. The last statement holds with $C_1:=\frac{1}{\varrho^{-}}$.
\end{proof}

\begin{lemma} \label{lemmaEM}
Let $f \in  D(\mathcal{E}) \cap C(\overline{\Omega})$ and choose a sequence $(f_k)_{k \in \mathbb{N}}$ in $\mathcal{D}$ whiches converges to $f$ with respect to $\mathcal{E}_1$. Then
\[ \nu_{\langle f_k \rangle}= |\nabla f_k |^2 ~\alpha \lambda + | \nabla_{\Gamma} f_k |^2 ~\beta \sigma \]
and $| \nabla_{\Gamma} f_k |^2=| \nabla f_k |^2 - | nn^t~ \nabla f_k|^2$ for each $k \in \mathbb{N}$ . Moreover, $(\nabla f_k)_{k \in \mathbb{N}}$ has the limit $\nabla f$ in $L^2(\overline{\Omega};\alpha \lambda)$ and similarly, $(\nabla_{\Gamma} f_k)_{k \in \mathbb{N}}$ has the limit $\nabla_{\Gamma} f \in L^2(\Gamma;\beta \sigma)$. In particular the convergence holds in $L^2_{\text{loc}}(\Omega \backslash \{ \varrho=0\};\lambda)$ and $L^2_{\text{loc}}(\Gamma \backslash \{ \varrho=0\};\sigma)$. The energy measure of $f$ is given by
\[ \nu_{\langle f \rangle}= |\nabla f |^2 ~\alpha \lambda + | \nabla_{\Gamma} f |^2 ~\beta \sigma. \]
\end{lemma}

\begin{proof}
Let $f \in \mathcal{D}$. Define $\nu:=|\nabla f |^2 ~\alpha \lambda + | \nabla_{\Gamma} f |^2 ~\beta \sigma$. We have to show that 
\[ 2 ~\mathcal{E}(fg,f)-\mathcal{E}(f^2,g)= \int_{\overline{\Omega}} g ~d\nu \]
for all $g \in D(\mathcal{E}) \cap C(\overline{\Omega})$. Then the result follows by uniqueness of $\nu_{\langle f \rangle}$. Since also $\mathcal{D}$ is dense in $C(\overline{\Omega})$ with respect to $\Vert \cdot \Vert_{\sup}$, it is enough to restrict to functions $g \in \mathcal{D}$. In this case,
\begin{align*}
&2 ~\mathcal{E}(fg,f)-\mathcal{E}(f^2,g)\\
&= \int_{\Omega} (\nabla (fg), \nabla f)~ \alpha d\lambda +  \int_{\Gamma} (\nabla_{\Gamma} (fg),\nabla_{\Gamma} f) ~\beta d\sigma - \frac{1}{2} \int_{\Omega} (\nabla f^2, \nabla g)~ \alpha d\lambda - \frac{1}{2} \int_{\Gamma} (\nabla_{\Gamma} f^2,\nabla_{\Gamma} g) ~\beta d\sigma\\
&=\int_{\Omega} (\nabla (fg), \nabla f)~ \alpha d\lambda +  \int_{\Gamma} (\nabla_{\Gamma} (fg),\nabla_{\Gamma} f) ~\beta d\sigma - \int_{\Omega} (\nabla f,f~ \nabla g)~ \alpha d\lambda - \int_{\Gamma} (\nabla_{\Gamma} f,f~ \nabla_{\Gamma} g) ~\beta d\sigma \\
&=\int_{\Omega} g (\nabla f, \nabla f)~ \alpha d\lambda +  \int_{\Gamma} g (\nabla_{\Gamma} f,\nabla_{\Gamma} f) ~\beta d\sigma \\
&=\int_{\overline{\Omega}} g ~ d\nu.
\end{align*}
Note that $|\nabla_{\Gamma} f|^2=|(E-nn^t) \nabla f|^2=|\nabla f|^2 -|nn^t~ \nabla f|^2$. Replacing $f$ by $f_k$ yields the first statement.
By \cite[p.124]{FOT94} it holds
\[ \left\arrowvert \big( \int_{\overline{\Omega}} g d\nu_{\langle f \rangle} \big)^{\frac{1}{2}} -\big( \int_{\overline{\Omega}} g d\nu_{\langle f_k \rangle} \big)^{\frac{1}{2}} \right\arrowvert \leq \big( \int_{\overline{\Omega}} g d\nu_{\langle f-f_k \rangle} \big)^{\frac{1}{2}} \leq \sqrt{2 \Vert f \Vert_{\sup}~ \mathcal{E}(f-f_k,f-f_k)}. \]
Hence,
\begin{align*}
\int_{\overline{\Omega}} g ~d\nu_{\langle f \rangle} &= \lim_{k \rightarrow \infty} \int_{\overline{\Omega}} g ~d\nu_{\langle f_k \rangle} \\
&=\lim_{k \rightarrow \infty} \big( \int_{\Omega} g ~|\nabla f_k|^2~ \alpha d\lambda +  \int_{\Gamma} g~ |\nabla_{\Gamma} f_k|^2 ~\beta d\sigma \big).
\end{align*}
Define $G_j :=\overline{\Omega} \backslash \overline{B_{\frac{1}{j}}(\{ \varrho=0\})}$ for $j \in \mathbb{N}$. Then each $G_j$ fulfills the assumptions of Lemma \ref{lemweakdiff} and $G_j \uparrow \overline{\Omega} \backslash \{ \varrho=0\}$ as $j \rightarrow \infty$. This yields a weak gradient $\nabla f$ and $\nabla_{\Gamma} f$ on each set $G_j$ and $G_j \cap \Gamma$ respectively. Therfore, we can define $\nabla f$ and $\nabla_{\Gamma} f$ globally outside $\{ \varrho=0\}$ and
\[ \int_{\Omega} |\nabla f|^2 ~ \alpha d\lambda \leq \liminf_{j \rightarrow \infty} \int_{\Omega} \mathbbm{1}_{G_j} |\nabla f|^2 ~\alpha d\lambda \leq \mathcal{E}_{\Omega}(f,f), \]
since the last inequality holds for fixed $j \in \mathbb{N}$. The statement holds similarly for $\nabla_{\Gamma} f$. Applying this to $f-f_k$ finishes the proof.
\end{proof}

\begin{proposition} \label{thmstronglyreg}
$(\mathcal{E},D(\mathcal{E}))$ is strongly regular.
\end{proposition}

\begin{proof}
We show that the intrinsic metric $d$ is equivalent to the euclidean metric $d_{\text{euc}}$. First, let $f_i(x):=x_i$, $x \in \overline{\Omega}$, for $i=1,\dots,d$. Then $f_i \in \mathcal{D}$ with $\frac{d\nu_{\langle f_i \rangle}}{d\mu} \leq 1$ a.e. and for $x,y \in \overline{\Omega}$ holds (by eventually replacing $f_i$ by $-f_i$)
\[ d(x,y) \geq \max_{i=1,\dots,d} \big(f_i(x)-f_i(y)\big) = \max_{i=1,\dots,d} |x_i-y_i| \geq \tilde{C}_1~ d_{\text{euc}}(x,y) \]
for some constant $\tilde{C}_1=\tilde{C}_1(d)< \infty$.
Moreover, by Lemma \ref{lemmaEM}
\begin{align*} d(x,y) &\leq \sup \{ f(x)-f(y) | \ f \in D(\mathcal{E}) \cap C(\overline{\Omega}) \text{ with } \nu_{\scriptscriptstyle{\langle f\rangle}} \leq \mu \} \\
&\leq \sup \{ f(x)-f(y) | \ f \in H^{1,\infty}(\Omega) \cap C(\overline{\Omega}) \text{ with } |\nabla f| \leq 1 \text{ a.e.} \} 
\end{align*}
and the last expression is locally bounded by $d_{\text{euc}}$. Indeed, by the proof of \cite[Satz 8.5]{Alt06} every $f\in H^{1,\infty}(\Omega)$ has a unique continuous version in $C^{0,1}(\overline{\Omega})$ and there is some constant $\tilde{C}_2=\tilde{C}_2(\Omega) < \infty$ such that 
\[ f(x)-f(y) \leq \tilde{C}_2 ~\Vert \nabla f \Vert_{L^{\infty}(\Omega)} ~d_{\text{euc}}(x,y). \]
\end{proof}


\begin{example} \label{excap}
Assume additionally to Condition \ref{conddensity} and Condition \ref{condhamza} that $\alpha, \beta \in C(\overline{\Omega})$ and the following property:
\begin{align} \label{condsturm} \mu(B_r(\{\varrho=0\})) \leq C ~r^2 \ \text{ as } r \rightarrow 0. \end{align}
Then, as a consequence of strong regularity, $\text{cap}_{\mathcal{E}}(\{ \varrho=0\})=0$ by \cite[Theorem 3]{Stu95} and therefore, Theorem \ref{thmergo} applies.
\end{example}

\section{$\mathcal{L}^p$-strong Feller properties} \label{sectfeller}

The diffusion process constructed in Section \ref{secSDE} has the drawback that the main result given in Theorem \ref{thmsolSDE} only holds for quasi every starting point $x \in \overline{\Omega}$ and it is not explicitly know how this set of admissible starting points looks like. In the following, we prove regularity properties of the associated $L^p$-resolvent and conclude that the results of Theorem \ref{thmsolSDE} even hold for every starting point $x \in \overline{\Omega}_1:=\overline{\Omega} \backslash \{ \varrho =0\}$ under Condition \ref{condcont} and Condition \ref{condcomp}. More precisely,
we show the sufficient conditions given in \cite[Condition 1.3]{BGS13} by generalizing a regularity result from \cite{Nit11} for $\delta=0$ and from \cite{War13} for $\delta=1$. Then, we apply \cite[Theorem 1.4]{BGS13}. This allows to proceed again as in Section \ref{secSDE}, but now without a set of starting points we have to exclude. Note that $\overline{\Omega}_1$ is not closed in $\mathbb{R}^d$ if $\{ \varrho =0\} \neq \emptyset$. We use this notation in order to be consitent with \cite{BGS13}.\\

We denote by $(T_t)_{t \geq 0}$ the strongly continuous contraction semigroup, by $(G_{\lambda})_{\lambda >0}$ the strongly continuous contraction resolvent and by $(L,D(L))$ the generator corresponding to $(\mathcal{E},D(\mathcal{E}))$. By the Beurling-Deny theorem there exists an associated strongly continuous contraction semigroup $(T_t^r)_{t >0}$ on $L^r(\overline{\Omega};\mu)$ with generator $(L_r,D(L_r))$ and resolvent $(G^r_{\lambda})_{\lambda >0}$ for every $1 \leq r < \infty$, see \cite[Proposition 1.8]{LS96} and \cite[Remark 1.3]{LS96}. If $r>1$ then $(T_t^r)_{t >0}$ is the restriction of an analytic semigroup by \cite[Remark 1.2]{LS96}. In this context associated means that for $f \in L^1(\overline{\Omega};\mu) \cap L^{\infty}(\overline{\Omega};\mu)$, it holds that $T_t f =T_t^r f$ for every $t >0$. With this notation we also have $T_t=T_t^2$ for $t \geq 0$, $G_{\lambda}= G_{\lambda}^2$ for $\lambda >0$ and $L_2=L$.\\

Assume that Condition \ref{condcont} is fulfilled. In order to prove the required regularity result we assume additionally the following property:

\begin{condition} \label{condint}
There exists $p \geq 2$ with $p > \frac{d}{2}$ such that
\[ \frac{|\nabla \alpha|}{\alpha} \in L^p_{\text{loc}}(\overline{\Omega} \cap \{ \varrho>0\}; \alpha \lambda) ~ \text{ and additionally } ~ \frac{|\nabla \beta|}{\beta} \in L^p_{\text{loc}}(\Gamma \cap \{ \varrho>0\}; \beta \sigma) \  \text{if } \delta=1 \]
or equivalently
\[ \mathbbm{1}_{\Omega} ~ \frac{|\nabla \alpha|}{\alpha} + \delta~ \mathbbm{1}_{\Gamma} ~ \frac{|\nabla \beta|}{\beta} \in L^p_{\text{loc}}(\overline{\Omega} \cap \{ \varrho>0\}; \mu). \]
\end{condition}

In the following, we assume Condition \ref{condcont}, Condition \ref{condint} and again that
\begin{enumerate}
\item $\text{cap}_{\mathcal{E}}( \{ \varrho =0\} ) =0$ (i.e., Condition \ref{condcomp}),
\end{enumerate}
which is e.g. fulfilled under the condition (\ref{condsturm}) given in Example \ref{excap}, and we prove that
\begin{enumerate}
\item[(ii)] there exists $p >1$ such that $D(L_p) \hookrightarrow C(\overline{\Omega}_1)$ and the embedding is locally continuous, i.e., for $x \in \overline{\Omega}_1$ there exists a $\overline{\Omega}_1$-neighborhood $U$ and a constant $C=C(U)< \infty$ such that
\[ \sup_{y \in U} |\tilde{u}| \leq C \Vert u \Vert_{D(L_p)} \ \text{ for all } u \in D(L_p), \]
where $\tilde{u}$ denotes the continuous version of $u$ (on $\overline{\Omega}_1$),
\item[(iii)] for each point $x \in \overline{\Omega}_1$ exists a sequence of function $(u_n)_{n \in \mathbb{N}}$ in $D(L_p)$ such that for every $y \neq x$, $y \in \overline{\Omega}_1$, exists a $u_n$ with $u_n(y)=0$ and $u_n(x)=1$.
\end{enumerate}
We say that a sequence $(u_n)_{n \in \mathbb{N}}$ as in (iii) is point separating in $x$. \\

Then, as a consequence of \cite[Theorem 1.4]{BGS13}, there exists a diffusion process
\begin{align*}
\mathbf{M}:=\big( \mathbf{\Omega}, \mathcal{F}, (\mathcal{F}_t)_{t \geq 0}, (\mathbf{X}_t)_{t \geq 0}, (\Theta_t)_{t \geq 0}, (\mathbf{P}_x)_{x \in \overline{\Omega}} \big)
\end{align*} with state space $\overline{\Omega}$ which leaves $\overline{\Omega}_1$ $\mathbf{P}_x$-a.s., $x \in \overline{\Omega}_1$, invariant. The Dirichlet form assciated to $\mathbf{M}$ is given by $(\mathcal{E},D(\mathcal{E}))$ and the transition semigroup $(p_t)_{t >0}$ of $\mathbf{M}$ is $\mathcal{L}^p$-strong Feller, i.e., $p_t(\mathcal{L}^p(\overline{\Omega};\mu)) \subset C(\overline{\Omega}_1)$. Moreover, it solves the $(L_p,D(L_p))$ martingale problem for every point $x \in \overline{\Omega}_1$. \\

\begin{lemma} \label{remgenerator}
For $p$ as in Condition \ref{condint} and $f \in C^2(\overline{\Omega})$ holds
\[ L_p f= L_2 f= L f \]
and in this case $L f$ is explicitly given by (\ref{generator}).
\end{lemma}

\begin{proof}
The statement for $p=2$ has been proven in Proposition \ref{propgen}. Then, the general statement follows by the assumptions on $\alpha$ and $\beta$ similar to Lemma 2.3 in \cite{BG12}, since $u$ and $Lu$ are elements of $L^p(\overline{\Omega};\mu)$ for $u \in C^2(\overline{\Omega})$.
\end{proof}

In a similar way as in the case of Neumann boundary conditions (see \cite[Section 4]{BG12}) we get the following:

\begin{theorem} \label{resolventeqthm}
Assume that Condition \ref{condcont} is fulfilled. Let $U$ be an open subset of $\overline{\Omega}$ in the subspace topology. The following holds:
\begin{enumerate} \item $C_c^1(U) \hookrightarrow \mathcal{D} \hookrightarrow D(\mathcal{E}).$
\item Assume additionally that $\overline{U} \subset \overline{\Omega}_1=\overline{\Omega} \backslash \{ \varrho =0\}$. The restriction maps $i_{\Omega}$ amd $i_{\Gamma}$ (supposed that $U \cap \Gamma \neq \emptyset$ and $\delta=1$), which restrict functions from $\overline{\Omega}$ to $U \cap \Omega$ and $U \cap \Gamma$ respectively, are continuous mappings from $D(\mathcal{E})$ to $H^{1,2}(U \cap \Omega)$ and $H^{1,2}(U \cap \Gamma)$ respectively. Moreover, it holds
\begin{align} \label{Eext} \mathcal{E}(u,v)= \frac{1}{2} \int_{U \cap \Omega} (\nabla u, \nabla v) ~\alpha d\lambda + \frac{\delta}{2} \int_{U \cap \Gamma} (\nabla_{\Gamma} u, \nabla_{\Gamma} v)~ \beta d\sigma 
\end{align} 
and there exists a constant $C_2=C_2(\alpha,\beta,d,G) < \infty$ such that 
\begin{align} \label{estimatesob} \Vert u \Vert_{H^{1,2}(U \cap \Omega)}^2 + \delta \Vert u \Vert_{H^{1,2}(U \cap \Gamma)}^2 \leq C_2 \mathcal{E}_1(u,u) 
\end{align}
for $u \in D(\mathcal{E})$ and $v \in C_c^1(U)$.
\item Let $2 \leq p < \infty$, $\gamma >0$. Let $x \in \overline{\Omega}$ and let $U := B_{R}(x)=\{ y \in \overline{\Omega}|~d_{\text{euc}}(x,y)<R\}$ be an open ball around $x$ in $\overline{\Omega}$ with radius $R >0$ such that $\overline{U} \subset \overline{\Omega}_1$. For all $f \in L^p(\overline{\Omega};\mu)$, we have $G_{\gamma}^p f  \in H^{1,2}(U \cap \Omega)$ and $G_{\gamma}^p f  \in H^{1,2}(U \cap \Gamma)$ for $\delta=1$, whenever $U \cap \Gamma$ is non-empty. Moreover, with $u:=G_{\gamma}^p f$ it holds
\begin{align} \label{resolventeq}
\gamma  \int_{U} u v ~d\mu + \frac{1}{2} \int_{U \cap \Omega} (\nabla u, \nabla v)~ \alpha d\lambda + \frac{\delta}{2} \int_{U \cap \Gamma} (\nabla_{\Gamma} u, \nabla_{\Gamma} v)~ \beta d\sigma = \int_{U} f v ~d\mu
\end{align}
for all $v \in C_c^1(U)$.\\
Additionally, for $R_0 >R$ such that $\overline{U_0} \subset \overline{\Omega} \backslash \{ \varrho =0 \}$, where $U_0 := B_{R_0}(x)$, we have the norm inequalities
\begin{align} \label{ineq1} \Vert u \Vert_{H^{1,2}(U \cap \Omega)}^2 + \delta \Vert u \Vert_{H^{1,2}(U \cap \Gamma)}^2 \leq C_3 ( \Vert f \Vert_{L^p(U_0; \lambda + \sigma)} + \Vert u \Vert_{L^p(U_0; \lambda + \sigma)})^2 \end{align}
and 
\begin{align} \label{ineq2} \Vert u \Vert_{H^{1,2}(U \cap \Omega)}^2 + \delta \Vert u \Vert_{H^{1,2}(U \cap \Gamma)}^2  \leq C_4 \Vert f \Vert^2_{L^p(\overline{\Omega};\mu)} \end{align}
\end{enumerate}
with constants $C_3=C_3(\alpha,\beta,R,R_0,d,p) < \infty$ and $C_4 =2~C_3< \infty$.
\end{theorem}

\begin{proof}
(i) is clear. The first part of (ii)  and inequality (\ref{estimatesob}) hold by Lemma \ref{lemweakdiff} (the result for $\delta=0$ holds similarly). 
\[ \mathcal{E}(u,v)= \frac{1}{2} \int_{U \cap \Omega} (\nabla u, \nabla v) ~\alpha d\lambda + \frac{\delta}{2} \int_{U \cap \Gamma} (\nabla_{\Gamma} u, \nabla_{\Gamma} v)~ \beta d\sigma \]
is evident for $u \in \mathcal{D}$ and $v \in C_c^1(U) \subset \mathcal{D}=C^1(\overline{\Omega})$. Fix $v \in C_c^1(U)$. Then $\mathcal{E}(\cdot,v)$ is a continuous linear functional on $D(\mathcal{E})$ with respect to the $\mathcal{E}_1^{\frac{1}{2}}$-norm. Moreover,  
\[ F(u):= \frac{1}{2} \int_{U \cap \Omega} (\nabla u, \nabla v) ~\alpha d\lambda + \frac{\delta}{2} \int_{U \cap \Gamma} (\nabla_{\Gamma} u, \nabla_{\Gamma} v)~ \beta d\sigma \]
is continuous on $D(\mathcal{E})$ (or rather on the space obtained by restricting functions to $U$) with respect to the norm given by $\Vert u \Vert_{H^{1,2}(U \cap \Omega)}^2 + \delta \Vert u \Vert_{H^{1,2}(U \cap \Gamma)}^2$, since $\alpha$ and $\beta$ are bounded from above and from below away from zero on $U$ (by continuity). Thus, it is also continuous with respect to the $\mathcal{E}_1^{\frac{1}{2}}$-norm in view of (\ref{estimatesob}) and therefore, $F$ has to coincide with $\mathcal{E}(\cdot,v)$ by uniqueness, since the equality holds on the dense subset $\mathcal{D}$. Therefore, (\ref{Eext}) is established.\\
Next, we prove (iii). Let $R$ and $R_0$ be as stated. First, we show (\ref{ineq1}) for $p=2$. Choose a cutoff function $\eta$ which is constantly one in $B_{R^{\prime}}(x)$ for some $R_0 > R^{\prime} > R$ and has compact support in $B_{R_0}(x)$. For $f \in L^2(\overline{\Omega};\mu)$ we have $u:=G_{\gamma}^2 f \in D(\mathcal{E})$ and it is easy to see that also $\eta u \in D(\mathcal{E})$, since $\eta u_n$ converges to $\eta u$ as $n \rightarrow \infty$ if $(u_n)_{n \in \mathbb{N}}$ approximates $u$ in $D(\mathcal{E})$. As in (ii) it can be shown that for fixed $v \in D(\mathcal{E})$ holds
\[ \mathcal{E}(v, \eta u)=\frac{1}{2} \int_{U_0 \cap \Omega} (\nabla v, \nabla (\eta u)) ~\alpha d\lambda + \frac{\delta}{2} \int_{U_0 \cap \Gamma} (\nabla_{\Gamma} v, \nabla_{\Gamma} (\eta u)) ~\beta d\sigma. \]
Note that $\eta^2$ is again a cutoff function with the properties we supposed for $\eta$. We have by calculation
\begin{align}
\mathcal{E}_{\gamma}(\eta u, \eta u)=&\gamma \int_{U_0} (\eta u)^2 d\mu + \frac{1}{2} \int_{U_0 \cap \Omega} (\nabla (\eta u), \nabla (\eta u)) ~\alpha d\lambda + \frac{\delta}{2} \int_{U_0 \cap \Gamma} (\nabla_{\Gamma} (\eta u), \nabla_{\Gamma} (\eta u)) ~\beta d\sigma \notag \\
=& \mathcal{E}_{\gamma}(u, \eta^2 u) - \frac{1}{2} \int_{U_0 \cap \Omega} \eta u (\nabla  u, \nabla \eta) ~\alpha d\lambda - \frac{\delta}{2} \int_{U_0 \cap \Gamma} \eta u (\nabla_{\Gamma}  u, \nabla_{\Gamma} \eta) ~\beta d\sigma  \notag \\
&+ \frac{1}{2} \int_{U_0 \cap \Omega}  u (\nabla  \eta, \nabla (\eta u)) ~\alpha d\lambda + \frac{\delta}{2} \int_{U_0 \cap \Gamma} u (\nabla_{\Gamma}  \eta, \nabla_{\Gamma} (\eta u)) ~\beta d\sigma \notag \\
=& \int_{U_0} f \eta^2 u~ d\mu - \frac{1}{2} \int_{U_0 \cap \Omega} \eta u (\nabla  u, \nabla \eta) ~\alpha d\lambda - \frac{\delta}{2} \int_{U_0 \cap \Gamma} \eta u (\nabla_{\Gamma}  u, \nabla_{\Gamma} \eta) ~\beta d\sigma \notag \\
&+ \frac{1}{2} \int_{U_0 \cap \Omega}  u (\nabla  \eta, \nabla (\eta u)) ~\alpha d\lambda + \frac{\delta}{2} \int_{U_0 \cap \Gamma} u (\nabla_{\Gamma}  \eta, \nabla_{\Gamma} (\eta u)) ~\beta d\sigma. \label{reform}
\end{align}
We get with the inequality $ a b \leq \frac{\varepsilon}{2} b^2+ \frac{1}{2 \varepsilon} a^2$ for $\varepsilon >0$, $a,b \geq 0$:
\begin{align*}
\left\arrowvert \int_{U_0 \cap \Omega} \eta u (\nabla  u, \nabla \eta) ~\alpha d\lambda \right\arrowvert &\leq K_1 \int_{U_0 \cap \Omega}  |\eta \nabla  u| |\nabla \eta | |u| ~ d\lambda \\
&\leq K_2 \int_{U_0 \cap \Omega}  |\eta \nabla  u| |u| ~  d\lambda \\
&\leq K_2 \big( \int_{U_0 \cap \Omega}  |\nabla (\eta u)| |u| ~ d\lambda + \int_{U_0 \cap \Omega}  |u| | \nabla  \eta| |u| ~ d\lambda \big) \\ 
&\leq K_3 \big( \Vert \nabla (\eta u) \Vert_{L^2(U_0 \cap \Omega; \lambda)}~ \Vert u \Vert_{L^2(U_0 \cap \Omega; \lambda)} + \Vert u \Vert_{L^2(U_0 \cap \Omega; \lambda)}^2 \big) \\
&\leq \frac{\varepsilon}{2} \Vert \eta u \Vert_{H^{1,2}(U_0 \cap \Omega)}^2 + (\frac{K_3^2}{2 \varepsilon} + K_3) \Vert u \Vert_{L^2(U_0 \cap \Omega; \lambda)}^2
\end{align*}
for suitable constants $K_1 \leq K_2 \leq K_3 < \infty$.
Similarly, we get (by eventually increasing $K_3$) that
\begin{align*}
\left\arrowvert \int_{U_0 \cap \Omega}  u (\nabla  \eta, \nabla (\eta u)) ~\alpha d\lambda \right\arrowvert \leq \frac{\varepsilon}{2} \Vert \eta u \Vert_{H^{1,2}(U_0 \cap \Omega)}^2 + (\frac{K_3^2}{2 \varepsilon} + K_3) \Vert u \Vert_{L^2(U_0 \cap \Omega; \lambda)}^2.
\end{align*}
For the two corresponding terms in (\ref{reform}) on $U_0 \cap \Gamma$, the similar statement follows by the same arguments. Moreover, we have
\begin{align*}
 \left\arrowvert \int_{U_0} f \eta^2 u~ d\mu \right\arrowvert &\leq \Vert f \Vert_{L^2(U_0;\mu)} \Vert \eta u \Vert_{L^2(U_0;\mu)} \\
&\leq \frac{1}{2} \big( \Vert f \Vert_{L^2(U_0;\mu)}^2 + \Vert \eta u \Vert_{L^2(U_0;\mu)}^2 \big) \\
&\leq K_4 \big( \Vert f \Vert_{L^2(U_0;\lambda + \sigma)}^2   + \Vert \eta u \Vert_{L^2(U_0;\lambda + \sigma)}^2 \big)
\end{align*}
for a constant $K_4 < \infty$.
Together with (\ref{estimatesob}) and (\ref{reform}) follows that there exists a constant $K_5 < \infty$ such that 
\begin{align*}
&\Vert \eta u \Vert_{H^{1,2}(U_0 \cap \Omega)}^2 +\delta \Vert \eta u \Vert_{H^{1,2}(U_0 \cap \Gamma)}^2 \\
&\leq  K_5 \big( \Vert f \Vert_{L^2(U_0;\lambda + \sigma)}^2 + (1 + \frac{1}{\varepsilon}) \Vert u \Vert_{L^2(U_0;\lambda + \sigma)}^2 + \varepsilon (\Vert \eta u \Vert_{H^{1,2}(U_0 \cap \Omega)}^2 + \delta \Vert \eta u \Vert_{H^{1,2}(U_0 \cap \Gamma)}^2) \big)
\end{align*}
Choosing $\varepsilon=\frac{1}{2K_5}$ yields a constant $K_6 < \infty$ such that
\begin{align*}
\Vert u \Vert_{H^{1,2}(U \cap \Omega)}^2 + \delta  \Vert u \Vert_{H^{1,2}(U \cap \Gamma)}^2 
\leq  K_6 \big(  \Vert f \Vert_{L^2(U_0;\lambda + \sigma)}^2 + \Vert u \Vert_{L^2(U_0;\lambda + \sigma)}^2 \big).
\end{align*}
For arbitrary $p \geq 2$ note that for $W:=L^1(\overline{\Omega};\mu) \cap L^{\infty}(\overline{\Omega};\mu) \subset L^2 (\overline{\Omega};\mu) \cap L^p(\overline{\Omega};\mu)$, $W$ is dense in $L^p(\overline{\Omega};\mu)$ and $G_{\gamma}^p f= G_{\gamma}^2 f$ for $f \in W$. For $f \in W$ inequality (\ref{ineq1}) applies, since the $L^2$-norm on $U_0$ can be estimated by the $L^p$-norm. Then (\ref{ineq1}) holds also for each $f \in L^p(\overline{\Omega};\mu)$ by a density argument and continuity of $G_{\gamma}^p$. (\ref{ineq2}) is a direct consequence of (\ref{ineq1}) and the fact that $G_{\gamma}^p$ is a contraction. \\
It rests to prove (\ref{resolventeq}). For $f \in W$ and $v \in C_c^1(U)$ holds $\mathcal{E}_{\gamma}(G_{\gamma}^2 f, v)=(f,v)_{L^2(\overline{\Omega};\mu)}$, i.e.,
\[ \gamma  \int_{U} G_{\gamma}^2 f v ~d\mu + \frac{1}{2} \int_{U \cap \Omega} (\nabla G_{\gamma}^2 f, \nabla v)~ \alpha d\lambda + \frac{\delta}{2} \int_{U \cap \Gamma} (\nabla_{\Gamma} G_{\gamma}^2 f, \nabla_{\Gamma} v)~ \beta d\sigma = \int_{U} f v ~d\mu \]
by (ii). Fix $v \in C_c^1(U)$ and let $f \in L^p(\overline{\Omega};\mu)$. Then we can approximate $f$ in $L^p(\overline{\Omega};\mu)$ by functions from $W$ due to density. Using (\ref{ineq2}) and continuity of the considered functionals, this proves (\ref{resolventeq}).
\end{proof}

\begin{corollary} \label{cororesolventeq}
Assume that Condition \ref{condcont} is fulfilled. Let $2 \leq p < \infty$, $\gamma >0$. Furthermore, let $x \in \overline{\Omega}$ and $U := B_{R}(x)=\{ y \in \overline{\Omega}|~d_{\text{euc}}(x,y)<R\}$ be an open ball around $x$ in $\overline{\Omega}$ with radius $R >0$ such that $\overline{U} \subset \overline{\Omega}_1$. For $u:= G_{\gamma}^p f$ holds
\begin{align} \label{resolventeq2}
\gamma  \int_{U} u v ~d\mu + \frac{1}{2} \int_{U \cap \Omega} (\nabla u, \nabla v)~ \alpha d\lambda + \frac{\delta}{2} \int_{U \cap \Gamma} (\nabla_{\Gamma} u, \nabla_{\Gamma} v)~ \beta d\sigma = \int_{U} f v ~d\mu
\end{align}
for all $v \in \mathcal{K}$, where $\mathcal{K}$ is defined as the closure of $C_c^1(U)$ with respect to the norm given by 
\[ \Vert \cdot \Vert_{\mathcal{K}}^2 := \Vert \cdot \Vert_{H^{1,2}(U \cap \Omega)}^2 + \delta \Vert \cdot \Vert_{H^{1,2}(U \cap \Gamma)}^2. \]
\end{corollary}

\begin{proof}
We fix $f \in L^p(\overline{\Omega};\mu)$ and $u=G_{\gamma}^p f$. Then, (\ref{resolventeq2}) yields continuous linear functionals on $\mathcal{K}$ and therefore, the assertion holds by density and (\ref{resolventeq}).
\end{proof}

\begin{remark}
We want to deduce from (\ref{resolventeq2}) that $G_{\gamma}^p f$ is continuous on $\overline{\Omega}_1$ for $p > \frac{d}{2}$. Note that for interior points $x \in \Omega \backslash \{ \varrho=0\}$ it is possible to choose $R$ small enough such that $B_R(x) \cap \Gamma = \emptyset$. In this case, (\ref{resolventeq2}) reduces to
\begin{align} \gamma  \int_{U} u v ~\alpha d\lambda + \frac{1}{2} \int_{U} (\nabla u, \nabla v)~ \alpha d\lambda  = \int_{U} f v ~\alpha d\mu \end{align}
for all $v \in H^{1,2}_0(U)$ with $f \in L^p(G;\lambda)$, $p > \frac{d}{2}$, i.e., this is the weak formulation of an elliptic PDE on $G$ with Dirichlet boundary conditions. Then it is well-known by the theory of DeGiorgi-Nash-Moser that $u$ is H{\"o}lder continuous near $x$ (see e.g. \cite{GT77}, \cite{Sta63} or \cite{HL97}). Thus, $x \in \Gamma \backslash \{ \varrho=0\}$ is the case of main interest.
\end{remark}

By Corollary \ref{cororesolventeq} $u:=G_{\gamma}^p f$ solves the equation (\ref{resolventeq2}). Therefore, the following theorem holds by an easy generalization of \cite{Nit11} for $\delta=0$ and \cite[Theorem 3.2]{War13} for $\delta=1$:

\begin{theorem} \label{thmcont}
Assume that Condition \ref{condcont} is fulfilled. Let $p > \frac{d}{2}$, $p \geq 2$, $\gamma >0$ and $f \in L^p(\overline{\Omega};\mu)$. Then $u:=G_{\gamma}^p f \in C(\overline{\Omega}_1)$ and for every $x \in \overline{\Omega}_1$ exists a neighborhood $U$ with $\overline{U} \subset  \overline{\Omega}_1$ and a constant $C_4=C_4(U,\alpha,\beta,d,p,\gamma)< \infty$ such that
\[ \sup_{y \in U} | \tilde{u}(y) | \leq C_4 \Vert f \Vert_{L^p(\overline{\Omega};\mu)}, \]
where $\tilde{u}$ denotes the continuous version of $u$ on $\overline{\Omega}_1$.
\end{theorem}

\begin{remark}
The regularity results in \cite{Nit11} and \cite{War13} (see also \cite{War12}) correspond to the special case of constant functions $\alpha$ and $\beta$. Nevertheless, the proofs generalize to our setting, since the densities $\alpha$ and $\beta$ are assumed to be continuous and therefore, they are locally on $\overline{\Omega_1}$ bounded from below away from zero.
\end{remark}

\begin{lemma} \label{lemmaps}
For each point $x \in \overline{\Omega}_1$ exists a sequence $(u_n)_{n \in \mathbb{N}}$ in $C^{\infty}(\overline{\Omega}) \subset D(L_p)$, $p >1$, that is point separating in $x$.
\end{lemma}

\begin{proof}
Fix $x \in \overline{\Omega}_1$ and $n \in \mathbb{N}$. Then it is clear that we can find a function $\tilde{u}_n$ in $C^{\infty}_c(\mathbb{R}^d)$ such that $\tilde{u}_n(x)=1$ and $\text{supp}(\tilde{u}_n) \subset B_{\frac{1}{n}}(x)$. Define $u_n:= \tilde{u}_n|_{\overline{\Omega}}$.
\end{proof}

\begin{theorem} \label{main} Assume that Condition \ref{condcont} and Condition \ref{condcomp} are fulfilled. Then there exists a conservative diffusion process
\[ \mathbf{M}=\big( \mathbf{\Omega}, \mathcal{F}, (\mathcal{F}_t)_{t \geq 0}, (\mathbf{X}_t)_{t \geq 0}, (\Theta_t)_{t \geq 0}, (\mathbf{P}_x)_{x \in \overline{\Omega}_1} \big)  \]
with state space $\overline{\Omega}_1$ such that
\begin{align*}
\mathbf{X}_t = x &+ \int_0^t \mathbbm{1}_{\Omega}(\mathbf{X_s}) dB_s + \int_0^t \mathbbm{1}_{\Omega} \frac{1}{2} \nabla \ln \alpha(\mathbf{X}_s) ds \\
&+ \delta \int_0^t \mathbbm{1}_{\Gamma}(\mathbf{X}_s) P(\mathbf{X}_s) dB_s - \delta \int_0^t  \mathbbm{1}_{\Gamma}(\mathbf{X}_s) \frac{1}{2} \kappa(\mathbf{X}_s) n(\mathbf{X}_s) ds \\
&+ \delta \int_0^t  \mathbbm{1}_{\Gamma}(\mathbf{X}_s) \frac{1}{2} \nabla_{\Gamma} \ln \beta (\mathbf{X}_s) ds - \int_0^t \frac{1}{2} \frac{\alpha}{\beta}(\mathbf{X}_s) n(\mathbf{X}_s) ds 
\end{align*}
almost surely under $\mathbf{P}_x$ for every $x \in \overline{\Omega}_1$. Moreover, its Dirichlet form is given by $(\mathcal{E},D(\mathcal{E}))$ on $L^2(\overline{\Omega}_1;\mu)$ and the transition semigroup $(p_t)_{t >0}$ of
$\mathbf{M}$ is $\mathcal{L}^p$-strong Feller, i.e., $p_t (\mathcal{L}^p (\overline{\Omega}_1;\mu)) \subset C(\overline{\Omega}_1)$. In particular, $(p_t)_{t >0}$ it strong Feller, i.e., $p_t (\mathcal{B}_b(\overline{\Omega}_1)) \subset C(\overline{\Omega}_1)$. Furhtermore, $\mathbf{M}$ has a sticky boundary behavior, i.e.,
\begin{align*} \lim_{t \rightarrow \infty} \frac{1}{t} \int_0^t \mathbbm{1}_{\Gamma}(\mathbf{X}_s) ds  >0 
\end{align*} 
$\mathbf{P}_x$-a.s. for every $x \in \overline{\Omega}_1$ such that $x$ is in a component of $\overline{\Omega}_1$ intersecting $\Gamma$.
\end{theorem}

\begin{proof}
First, we have to check the assumptions of \cite[Theorem 1.4]{BGS13}, namely that $D(L_p) \hookrightarrow C(\overline{\Omega}_1)$, that the embedding is locally continuous and the point separating property. It holds $D(L_p)= G_{\gamma}^p L^p(\overline{\Omega};\mu)$ and hence, we have $D(L_p) \hookrightarrow C(\overline{\Omega}_1)$ and moreover, for $u=G_1^p f \in D(L_p)$ it holds locally 
\begin{align*} \sup_{y \in U} | \tilde{u}(y) | &\leq C_4 \Vert f \Vert_{L^p(\overline{\Omega};\mu)} \\
&=C_4 \Vert (1-L_p)u \Vert_{L^p(\overline{\Omega};\mu)} \leq C_4 (\Vert u \Vert_{L^p(\overline{\Omega};\mu)} + \Vert L_p u \Vert_{L^p(\overline{\Omega};\mu)}) = C_4 \Vert u \Vert_{D(L_p)}.
\end{align*}
The existence of a point spearating sequence for each point $x \in \overline{\Omega}_1$ follows by Lemma \ref{lemmaps}. This assures the existence of a process $\mathbf{M}$ with state space $\overline{\Omega}$ as stated at the beginning of this section such that $\overline{\Omega}_1$ is invariant for all starting points in $\overline{\Omega}_1$ and its transition semigroup is $\mathcal{L}^p$-strong Feller. In particular, the process is a solution of the given SDE for every starting point $x \in \overline{\Omega}_1$. This follows by the fact that $\mathbf{M}$ solves the $(L_p,D(L_p))$ martingale problem and $L_p$ is given as in Proposition \ref{propgen} for functions in $C^2(\overline{\Omega})$ (see also Remark \ref{remgenerator}). Since $\mathcal{B}_b(\overline{\Omega}) \subset L^p(\overline{\Omega};\mu)$, it follows that the process is also strong Feller in the sense that the transition semigroup maps $\mathcal{B}_b(\overline{\Omega})$ into $C(\overline{\Omega}_1)$. By admitting only starting points in $\overline{\Omega}_1$ and invariance, we obtain a process $\mathbf{M}$ as stated. $\mathcal{L}^p(\overline{\Omega}_1; \mu) \tilde{=} \mathcal{L}^p(\overline{\Omega};\mu) \subset C(\overline{\Omega}_1)$ and therefore, the semigroup is ($\mathcal{L}^p$-)strong Feller. The associated Dirichlet form is given by $(\mathcal{E},D(\mathcal{E}))$ on $L^2(\overline{\Omega}_1;\mu)$ by Definition \ref{defpart} and the following remark on parts of processes. In particular, the absolute continuity condition given in \cite[(4.2.9)]{FOT94} is fulfilled and therefore, the ergodicity result holds accordingly for every starting point $x \in \overline{\Omega}_1$, since the required properties directly transfer from the $L^2(\overline{\Omega};\mu)$ to the $L^2(\overline{\Omega}_1;\mu)$ setting.
\end{proof}

\begin{remark}
In principle, it is also possible to define the Dirichlet form $(\mathcal{E},D(\mathcal{E}))$ on $L^2(\overline{\Omega}_1;\mu)$ in the first place. Then, proving elliptic regularity of the associated resolvent yields a strong Feller process associated to $(\mathcal{E},D(\mathcal{E}))$ without removing a set of capacity zero. In this case, $\{ \varrho=0\}$ is part of the boundary of the state space which requires to use Dirichlet boundary conditions on $\{ \varrho=0\}$. Thus, the assumption that $\{ \varrho=0\}$ is of capacity zero is in this case replaced by conservativeness. Admittedly, this procedure also allows the construction of a non-conservative solution which corresponds to the case of non-zero capacity.
\end{remark}

\subsection*{Acknowledgment}

R.~Vo{\ss}hall gratefully acknowledges financial support in the form of a fellowship of the German state Rhineland-Palatine.

\bibliographystyle{alpha}
\bibliography{biblio}

\end{document}